\documentclass[sn-mathphys-num]{sn-jnl}
\usepackage{graphicx}%
\usepackage{multirow}%
\usepackage{amsmath,amssymb,amsfonts}%
\usepackage{amsthm}%
\usepackage{mathrsfs}%
\usepackage[title]{appendix}%
\usepackage{xcolor}%
\usepackage{textcomp}%
\usepackage{manyfoot}%
\usepackage{booktabs}%
\usepackage{algorithm}%
\usepackage{algorithmicx}%
\usepackage{algpseudocode}%
\usepackage{listings}%
\usepackage{enumitem}%
\usepackage{mathtools}
\usepackage{tikz-cd}
\usepackage{hyperref}

\DeclarePairedDelimiter\floor{\lfloor}{\rfloor}
\theoremstyle{thmstyleone}%
\newtheorem{theorem}{Theorem}% 
\newtheorem{proposition}[theorem]{Proposition}% 
\newtheorem{lemma}[theorem]{Lemma}% 

\theoremstyle{thmstyletwo}%

\theoremstyle{thmstylethree}%
\newtheorem{definition}{Definition}%

\newcommand{\M}{\mathcal{M}}

\begin{document}

\title[Article Title]{A discrete-time temporal deontic STIT logic based on interpreted systems}

\author[1]{\fnm{Shuge} \sur{Rong}}
\author*[2]{\fnm{Yifeng} \sur{Ding}}\email{yf.ding@pku.edu.cn}

\affil[1]{\orgdiv{Department of Philosophy}, \orgname{Carnegie Mellon University}}
\affil[2]{\orgdiv{Department of Philosophy and Religious Studies}, \orgname{Peking University}}


\abstract{We present a STIT (`see to it that') logic with discrete temporal operators and deontic operators in which we can formalize and reason about legal concepts such as persistent duty and the dynamic concept of power from Hohfeld. As our main technical contribution, we show that this logic is sound and complete with respect to the semantics based on interpreted systems and is decidable.}

\keywords{STIT logic, temporal STIT logic, interpreted system, axiomatization, legal reasoning}

\maketitle

\section{Introduction}
In a world full of agents, we naturally assign certain propositions as the consequences of certain agents' actions and regard them as responsible. Among many colloquial ways to express such assignments, propositions of the form `$a$ sees to it that $\varphi$' have been argued to be the canonical form of such agential propositions \cite[Chapter 1]{belnap2001facing}, and theories and logics revolving it, often called `STIT logics', have been systematically developed and are highly influential \cite{belnap2001facing,hortyAgencyDeonticLogic2001a,schwarzentruber2012complexity,balbiani2008alternative,lyonAutomatingAgentialReasoning2019,broersenDeonticEpistemicStit2011,broersenMakingStartStit2011}. Our goal in this paper is to study a STIT logic including deontic operators and discrete-time temporal operators and show that it is sound and complete with respect to its natural semantics based on (synchronous) interpreted systems \cite{lehmannKnowledgeCommonKnowledge1984,halpern2004complete}. Our method is by filtration, and thus the decidability of this logic also follows.

Our interest in this logic is partly motivated by its potential to help us understand and extend the seminal analysis of legal relations by Hohfeld \cite{hohfeld1917fundamental}. Lamenting on the imprecise uses of the fundamental concepts such as `right' and `power' and the consequent failure to correctly understand more complex concepts based on these fundamental concepts, Hohfeld put forward a theory of eight fundamental legal relations that could serve to analyze precisely the difficult cases he mentioned and in principle the whole legal enterprise, which would allow for the discovery of deep unity and harmony there. The eight relations are divided into two groups and are shown in Figure \ref{fig:Hohfeldian-concepts}. The left group may be call the \emph{static} relations as they directly specify what one are legally permitted or required to do, while the right group may be called the \emph{dynamic} relations as they talk about people's power or lack of power to change legal relations in the left group. 
\begin{figure}
    \centering
    \begin{tikzcd}[column sep = 7em, row sep = large]
        \textrm{Claim} \arrow[r, leftrightarrow, "\textit{Correlatives}" description] \arrow[d, leftrightarrow, "\textit{Opposites}" description] & \textrm{Duty} \arrow[d, leftrightarrow, "\textit{Opposites}" description] \\
        \textrm{No-Claim} \arrow[r, leftrightarrow, "\textit{Correlatives}" description] & \textrm{Previlege}
    \end{tikzcd}\hspace{1em}
    \begin{tikzcd}[column sep = 7em, row sep = large]
        \textrm{Power} \arrow[r, leftrightarrow, "\textit{Correlatives}" description] \arrow[d, leftrightarrow, "\textit{Opposites}" description] & \textrm{Liability} \arrow[d, leftrightarrow, "\textit{Opposites}" description] \\
        \textrm{No-Power} \arrow[r, leftrightarrow, "\textit{Correlatives}" description] & \textrm{Immunity}
    \end{tikzcd}
    \caption{Hohfeld's fundamental legal relations}
    \label{fig:Hohfeldian-concepts}
\end{figure}
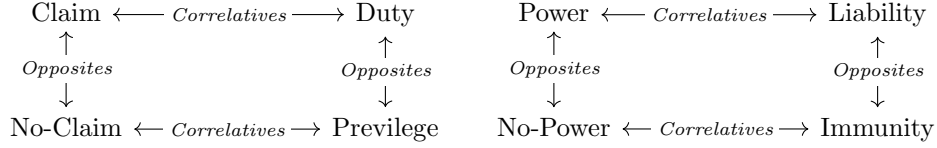

In the language of deontic STIT logic, the static relations can be formalized and rigorously studied \cite{lindahl1977position}. For example, that $a$ and $b$ are related by a `claim' relation such as `$a$ has a claim against $b$ that $b$ shall issue a public apology' can be formalized simply as $\mathsf{O}_b[b]p$ where $\mathsf{O}_b$ is the deontic `ought' operator for $b$, $[b]$ is the STIT operator for $b$, and $p$ designates the proposition that $b$ issues a public apology. That $b$ has a duty to $a$ to issue a public apology is formalized by exactly the same formula, and this is why the claim and the duty relations are correlatives: a claim that $\varphi$ from $a$ to $b$ \emph{just is} a duty that $\varphi$ from $b$ to $a$. In modern terminology in modal logic, no-claim is the negation of claim while privilege is the \emph{dual} of duty. For example, that $a$ has the privilege (against $b$) to publish $a$'s most recent paper is formalized as $\lnot \mathsf{O}_a [a] \lnot p$, which just is that $b$ does not have the claim that $a$ shall see to it that $a$'s most recent paper is not published. One might contend that $a$'s privilege to publish $a$'s most recent paper should instead say that $\lnot \mathsf{O}_a \lnot [a] p$, that is, $a$ `may' see to it that $a$'s most recent paper is published, where `may' (for agent $a$), which we will write as $\mathsf{M}_a$, is the dual of $\mathsf{O}_a$. The distinction will vanish if we assume that in `$a$ has privilege to $\varphi$', $\varphi$ must be agentive for $a$, meaning that $\varphi$ is already of the form $[a] \psi$ so that $[a] \lnot\varphi$ and $\lnot \varphi$ are logically equivalent under the usual $\mathsf{S5}$ logic for (Chellas) STIT. In any case, while Hohfeld's original system may still be lacking as seen from the remaining ambiguities under the lens of deontic STIT logic and the strange asymmetry between the opposition between claim and no-claim (which is negation) and the opposition between duty and privilege (which is duality), what is important is that deontic STIT logic can see through these problems and provide a rich theory of normative positions extending Hohfeld's four static relations. A clear introduction to the theory of normative position is \cite{sergotNormativePositions2013}.

One may wonder why the deontic operator $\mathsf{O}$ is subscripted by the agent name, since standard deontic logics do not have this feature, and so do many theories of normative positions. If we are fully faithful to Hohfeld, deontic operator that are doubly subscripted such as $\mathsf{O}_{ba}$ are more appropriate, where the second subscript $a$ denotes the source or the beneficiary of this obligation, as in if the obligation is not observed by $b$, then $a$ is wronged because of this, or $a$ should be compensated because of this. We chose not to use the doubly subscripted $\mathsf{O}_{ba}$ because if no special axioms are added regarding how to aggregate the $\mathsf{O}_{ba}$'s for different $a$'s into $\mathsf{O}_b$ that stands for the `all things considered' obligation for $b$, then such a generalization is technically trivial; but finding reasonable aggregation axioms is out the scope of this paper, and we will leave it for another occasion. On the other hand, we chose to include a deontic operator for each agent since we find this conceptually more faithful than using only a single deontic operator $\mathsf{O}$. The legal duties are always directed to agents, often irrespective of what other agents would do at the same time. Admittedly, a non-agent-specific $\mathsf{O}$ can provide extra expressivity since in our later formal developments, each $\mathsf{O}_a$ only talks about what actions $a$ must take, while $\mathsf{O}$ could specify what kind of coordinated acts the agents must simultaneously take; e.g., $\mathsf{O}$ may specify that the agents must play a Nash equilibrium in a coordination game. But it seems that the law should not use such extra expressivity. Mathematically, having multiple deontic operators shows that our proof method can handle more than one such operators, and using only $\mathsf{O}$ is mathematically only simpler. 

The four dynamic legal relations pose a more serious challenge to formalization, and there have been many discussions and proposals, including some very recent ones. The central idea of the dynamic concepts is the ability or inability to change legal relations. For example, when $a$ owns a smartphone $x$, $a$ not only has a claim against $b$ that $b$ shall not sell $x$ for money, $a$ also has the \emph{power} to extinguish this claim by gifting $x$ to $b$ or by simply abandoning $x$. To fully account for the dynamic nature of power, \cite{markovichUnderstandingHohfeldFormalizing2020} uses the public announcement operator, \cite{dongDynamicLogicLegal2021} uses the full event model dynamic operators, and \cite{vaneijckModelingDynamicsLegal2024} uses propositional dynamic logic with atomic actions that have normative effects. In our deontic STIT language extended with temporal operators, one way of approximating powers of $a$ is by $\Diamond[a]\mathsf{X}\varphi$ where $\Diamond$ is the possibility operator and $\mathsf{X}$ is the `next time' operator. Here $\mathsf{X}$ provides the dynamics while $\Diamond$ in effect quantifies over what $a$ can do. For example, when $a$ has an offer from $b$ to purchase a cellphone $x$ at price \$$y$, $a$ thereby has a power that can be expressed by $\Diamond [a] \mathsf{X} (\mathsf{O}_a [a] p_1 \land \mathsf{O}_b [b] p_2)$ where $p_1$ says that $a$ transfers \$$y$ to $b$ and $p_2$ says that $b$ transfers $x$ to $a$. There is indeed a change of obligations, since before $a$ accepts the offer, clearly there are no obligations for $a$ to transfer $b$ money or for $b$ to give $x$ to $a$. The $\Diamond$ operator is necessary here since having power is clearly distinct from executing that power. In \cite{belnap2001facing}, the version $\Diamond[a]$ without the $\mathsf{X}$ operator is proposed. In \cite{dongDynamicLogicLegal2021}, the existential quantification over actions is approximated by an explicit finitary disjunction over atomic actions in a given event model, which means that we can only say that `$a$ has such and such power in this event model', but never just `$a$ has such and such power'. \cite{markovichUnderstandingHohfeldFormalizing2020} and \cite{vaneijckModelingDynamicsLegal2024} do not allow quantification over actions in their formal language. 

Two disadvantages of $\Diamond [a] \mathsf{X} \varphi$ as a formalization of $a$'s power to $\varphi$ must immediately be acknowledged. First, the choice of $\mathsf{X}$ may seem arbitrary and even superfluous. Indeed, $\Diamond [a] \mathsf{X}\mathsf{X} \varphi$ may very well also describe a power of $a$, so there is indeed some arbitrariness. However, $\Diamond [a] \varphi$, in our semantics later, is fundamentally different, since we will be committed to the view that no matter what any agent does at this moment, what each agent ought to do at this moment is a fixed matter. In other words, no simultaneous game theory. If for example $\mathsf{O}_b p$, then $\Box \mathsf{O}_b p$ and thus $\lnot \Diamond [a] \lnot \mathsf{O}_b p$. As such, there is nothing $a$ can do to change the static legal relations between $a$ and $b$ at this moment, but it may well be that $\mathsf{X} \lnot \mathsf{O}_b p$ is true and this is brought about by a possible action of $a$. 

Second, one might contend that this formalization by $\Diamond [a] \mathsf{X}$ misses the target, because $\Diamond$ in STIT logic means physical possibility, and as pointed out by Hohfeld himself, the legal power of bringing about legal consequences when doing certain actions must be distinguished from the physical power to do those actions. Indeed, there is a prominent line of research focusing intensely on the logic of attaching normative content to actions without much regard for whether the actions can be carried out, and existential quantification over actions is not important. To approximate this sense of legal power, one idea is to understand $\Box$ in STIT more broadly and consider formulas of the form of $\Box\mathsf{G}'([a]\varphi \to \mathsf{X}\psi)$ where $\mathsf{G}'$ here means `true now and in all the future times'. When $\psi$ is a normative position, this means that $a$ has the power to bring about the position $\psi$ in the next moment by performing $\varphi$. Admittedly, this is still not as fine-grained as many of the previous analyses, but hopefully we have captured the most essential aspects of power that are of practical importance.

The interest of a temporal deontic STIT logic, of course, does not solely lie in the formalization of the Hohfeldian theory of legal relations. Any normative requirement with a deadline needs temporal operators to properly formalize. When the deadline is some fixed time in the future, some iteration of $\mathsf{X}$ can do the job. But one may also promise to do certain things in the future, without specifying a specific date, and in this case $\mathsf{F}\mathsf{O}_a [a] \varphi$ shall be used. Obligations involving the until operator $\mathsf{U}$ are also common; for example, when one borrows money without fixing the date of repaying the principal but promises to repay it in the future, the obligation is to pay the interest every year until you repay the principle, and $\mathsf{U}(q, \mathsf{O}_a [a] p)$ can be used to formalize this idea. \cite{governatoriNormComplianceBusiness2010} has more examples but deals with such obligations very differently. Other works on temporal obligations include \cite{broersenDesigningDeonticLogic2004,dignumMeetingDeadlineWhy2005,balbianiDecisionProceduresDeontic2009,demolombeObligationsDeadlinesFormalization2014}.

The temporal dimension has always been an important component of STIT theory and its underlying branching time worldview \cite{belnap2001facing}. More recently, Lorini \cite{lorini2013temporal} axiomatized a temporal STIT logic based on temporal STIT Kripke models and showed how obligations with temporal dimensions can be formalized in this logic using propositional variables that stand for `violations', and \cite{van2019neutral} extended this logic with explicit deontic operators for each agent. Ciuni and Lorini \cite{ciuni2018comparing} showed that the semantics based on these Kripke models are equivalent to the semantics based on the original branching-tree models, but with a distinguished bundle of histories. \cite{ciuni2018comparing} also studied discrete-time models and showed that the semantics based on Fagin et al.'s interpreted systems \cite{faginReasoningKnowledge2004,halpern2004complete}, which will be our semantics of choice, is equivalent to a number of other semantics, including the semantics based on discrete-time branching-tree models with distinguished bundles. The question of axiomatizing the temporal STIT logic under the semantics based on interpreted systems was also raised in \cite{ciuni2018comparing}. We solve this question with deontic operators added to the language.

Our method is based on filtrating the canonical model, as is common in the completeness proofs of propositional dynamic logics or any modal logic involving a fixed-point operator (the `until' operator in our case). Since our models are subject to strong structural conditions from STIT theory, we adopt the filtration method from \cite{gabbay2003many}. The proof also involves transforming the so-called super-additive models to additive models (a similar procedure is required to deal with distributed knowledge), and we draw inspiration from \cite{schwarzentruber2012complexity}. Since our method is based on filtration, the decidability of the logic follows. It should be mentioned that the 2-EXPTIME completeness of the satisfiability problem of the temporal STIT logic, without the grand coalition STIT operator and based on full discrete-time branching time models, has been obtained in \cite{boudouConcurrentGameStructures2018} using automata.

The rest of the article is organized as follows. In section \ref{sec:set-up}, we set up the formal language, introduce the semantics based on interpreted systems, and present the axiomatic system that is sound and complete with respect to this semantics. Then, we move on to prove the completeness, the idea being to filtrate the canonical model and then transform the result into an interpreted system. We first cover the transformations in Section \ref{sec:general-semantics}, including transforming a super-additive interpreted system to an equivalent additive one, and transforming a Kripke model satisfying suitable constraints to an equivalent super-additive interpreted system. We finish our completeness proof in Section \ref{sec:completeness} by constructing the canonical model and the appropriate filtration. We also remark on the decidability. We conclude in Section \ref{sec:conclusion}.

\section{Interpreted system for STIT and its Logic}\label{sec:set-up}
Let us first fix the formal language of our discrete-time temporal deontic STIT logic.
\begin{definition}
Let $\mathsf{Agt} = \{1,...,n\}$ be the non-empty finite set of agents and 
let $\mathsf{Prop}$ be the countably infinite set of propositional variables. 
The language $\mathcal{L}_\mathrm{DTDS}$ is given by the following BNF grammar:
\setlength{\abovedisplayskip}{3pt}
\setlength{\belowdisplayskip}{3pt}
\begin{align*}
    \varphi::=p\mid\neg\varphi\mid(\varphi\land\varphi)\mid\Box\varphi\mid [i]\varphi\mid [\mathsf{Agt}]\varphi\mid\mathsf{O}_i\varphi\mid \mathsf{X}\varphi\mid \mathsf{U}(\varphi,\varphi)
\end{align*}
\noindent where $i\in \mathsf{Agt}$, $p\in \mathsf{Prop}$.
The usual abbreviations apply; for example, 
$\Diamond := \lnot\Box\lnot$, 
$\langle i\rangle := \lnot [i] \lnot$, 
$\langle \mathsf{Agt} \rangle := \lnot [\mathsf{Agt}] \lnot$.
%$\mathsf{F}\varphi := \mathsf{U}(\top, \varphi)$, and 
%$\mathsf{G}\varphi := \lnot \mathsf{F} \lnot \varphi$.
\end{definition}

$\Box\varphi$ says that $\varphi$ is necessarily true 
regardless of how the world will develop differently, i.e., a historical necessity. 
$[i]$ is the `Chellas STIT' operator that allows for $[i]\varphi$ even when $\varphi$ 
is a historical necessity. 
$[\mathsf{Agt}]$ is the `Chellas STIT' operator for the whole group; 
$[\mathsf{Agt}]\varphi$ means that the joint action of the whole group 
`sees to it that' $\varphi$. 
$\mathsf{O}_i$ is the `ought' operator for $i$. 
We take it that different agents may be under different duties 
and thus introduce such an operator for each agent. 
We will also make sure that $\mathsf{O}_i\varphi$ is really saying that 
`$i$ ought to see to it that $\varphi$'.
$\mathsf{X}$ is the standard `next time' operator, while $\mathsf{U}$ 
is an `until' operator.
There are many candidate semantics for `until', and ours will be as follows:
$\mathsf{U}(\varphi, \psi)$ says that there is a future time point $t$ 
(including the current time) where $\psi$ is true, 
and for each time $t'$ from \emph{now} until \emph{before} $t$, 
$\varphi$ is true at $t'$. 
Other versions of `until' in discrete linear temporal logic can be defined
by this version of `until' together with `next time'. 

%We define the dual of $\Box$ as $\Diamond$. Notice that a moment is a specific time point shared by one or more histories. As shown in Figure \ref{fig:Histories and Moments}, $h_1,...,h_6$ are histories, and $m_1,m_2,m_3$ are moments in histories.

%\begin{figure}[h]
%    \centering
%\includegraphics[width=0.4\textwidth]{fig/1041697263102_.pic.jpg}
%    \caption{Histories and Moments}
%    \label{fig:Histories and Moments}
%\end{figure}

In the following, 
we introduce the interpreted systems-based semantics for 
$\mathcal{L}_{\mathrm{DTDS}}$ 
and the axiomatic system $\mathsf{L}_{\mathrm{DTDS}}$.

\subsection{Interpreted systems}
An interpreted system is essentially a possible world-based Kripke model with a more concrete story for each possible world: each world is temporally complex and consists of a history of stages indexed by $\mathbb{N}$, which allows us to naturally interpret $\mathsf{X}$ and $\mathsf{U}$.  A world is in a state at each stage, and two worlds are historically identical at stage $t$ if they were in identical states at all stages $t' < t$. For any world $w$ at stage $t$, it is well defined what action did each agent perform, and we can also tell what possible actions an agent $i$ can perform by looking at the actions performed by $i$ at stage $t$ in all possible worlds that are historically identical to $w$ at stage $t$. The signature thesis of STIT theory defended by its proponents is that every combination of possible actions for each agent is realized in a historically identical possible world. For the deontic operator, let us also suppose that at each stage of each possible world, it is well defined what possible actions each agent is deontically allowed to perform. The power of $i$ lies in the action she performs. If worlds $w$ and $w'$ are historically identical at stage $t$ but $i$ performs different actions in them at stage $t$, then in world $w$ at stage $t$, we can say that $i$ `sees to it that' she is not living in $w'$. In general, for $i$ to `see to it that' $\varphi$ is for $i$'s action to rule out all non-$\varphi$ worlds, or equivalently, for $\varphi$ to be a consequence of $i$'s performing the action she actually performs. Similarly, for the whole group $\mathsf{Agt}$ to `see to it that' $\varphi$ is for the whole group's joint action to force $\varphi$. 

Since our language $\mathcal{L}_{\mathrm{DTDS}}$ does not care about the internal structure of the states of the worlds or the actions each agent performs, but only whether and when the worlds are in the same states and whether and when the agents are performing the same actions or their allowed actions, it will suffice to use some binary relations to turn the above into mathematical models for $\mathcal{L}_{\mathrm{DTDS}}$. We introduce some notations before the formal definition.
\begin{definition}
    For any set $I$ and sets $\{X_i \mid i \in I\}$ indexed by $I$, 
    $\Pi_{i \in I}X_i$ is the set of all functions $f$ whose domain is $I$
    and for any $i \in I$, $f(i) \in X_i$. 

    For any binary relation $R$ on a set $X$ and any $Y \subseteq X$, 
    let $R[Y] = \{x \in X \mid \exists y \in Y, yRx\}$. 
    We also write $R[y]$ for $R[\{y\}]$. 
    Note that if $R$ is an equivalence relation, then $R[x]$ is the 
    equivalence class $x$ is in. 
    For any equivalence relation $R$ on $X$ and any subset $Y \subseteq X$, 
    $||R||_Y$ is the set of all equivalence classes of $R$ that are contained in $Y$.
    Note that when $Y$ is itself a union of some equivalence classes of $R$, 
    then $||R||_Y = \{R[y] \mid y \in Y\}$. 

    If $R, R'$ are binary relations, $R \circ R'$ is their composition: $x (R \circ R') y$ 
    iff there is $z$ such that $x R z$ and $z R' y$.
\end{definition}
\begin{definition}
    \label{def:interpreted-system}
    A \emph{DTDS interpreted system} (or just \emph{interpreted system}) is a tuple 
    $$
    \mathcal{I} = 
    (H, R_\Box, \{R_{[i]}, R_{\mathsf{O}_i} \mid i \in Agt\}, 
    R_{\mathsf{Agt}}, V)
    $$ 
    where
    \begin{itemize}
    \item $H$ is a non-empty set (the set of possible histories),
    \item $R_\Box$ is an equivalence relation on $H \times \mathbb{N}$ and 
    for each $i \in \mathsf{Agt}$, 
    $R_{[i]}$ is an equivalence relation on $H \times \mathbb{N}$ while 
    $R_{\mathsf{O}_i}$ is a binary relation on $H \times \mathbb{N}$,
    \item $V$ is a function from $\mathsf{Prop}$ to $\wp(H \times \mathbb{N})$,
    \item and the following special properties hold:
    \begin{enumerate}
        \item[(D0)] for any $(h, t), (h', t') \in H \times \mathbb{N}$, 
        if $(h, t)R_\Box(h', t')$ then $t = t'$;
        \item[(D1)] $R_{[i]} \subseteq R_\Box$;
        %\item[(D2)] if for each $i \in \mathsf{Agt}$, $(h_i, t) R_\Box (h', t)$, then there is $(h, t)$ such that $(h, t) R_\Box (h', t)$ and $(h', t)R_{[i]}(h_i, t)$ for each $i \in \mathsf{Agt}$;
        \item[(D2)] for every $M \in ||R_\Box||$ and 
            every $f \in \Pi_{i \in \mathsf{Agt}}||R_{[i]}||_M$, 
            $\bigcap_{i \in \mathsf{Agt}} f(i) \not= \varnothing$;
        \item[(D3)] $R_{\mathsf{Agt}} = \bigcap_{i \in Agt}R_{[i]}$.
        \item[(D4)] for every $h, h' \in H$ and $t \in \mathbb{N}$, 
        if $(h, t+1)R_\Box(h', t+1)$ then $(h, t)R_{\mathsf{Agt}}(h', t)$;
        \item[(D5)] $R_{\mathsf{O}_i}\subseteq R_{\Box}$;
        \item[(D6)] $R_{\mathsf{O}_i}$ is serial;
        \item[(D7)] $R_{\mathsf{O}_i} \circ R_{[i]} \subseteq R_{\mathsf{O}_i}$;
        \item[(D8)] $R_\Box \circ R_{\mathsf{O}_i} \subseteq R_{\mathsf{O}_i}$.
    \end{enumerate}
\end{itemize}
\end{definition}
To better explain the above conditions, for each binary relation $R_*$ in an DTDS interpreted system with the set of histories $H$ and each $t \in \mathbb{N}$, let us define $R^t_* = \{(h, h') \in H^2 \mid (h, t) R^t_* (h', t)\}$. Then (D0) says that $R_\Box$ can be sliced into $\{R^t_{\Box} \mid t \in \mathbb{N}\}$ and worlds at different stages cannot be historically identical. The equivalence classes of $R^t_{\Box}$ represent the possible developments of the worlds up to stage $t - 1$. Since we want $R_{[i]}$ to mean that two worlds are not only historically identical at stage $t$ but also that $i$ performs the same action at stage $t$,  (D1) is clearly necessary for the `not only' part.  It makes sure that each equivalence relation $R^t_{[i]}$  refines the equivalence relation $R^t_{\Box}$, and for each equivalence class $M$ of $R^t_{\Box}$, the equivalence classes of $R^t_{[i]}$ contained in $M$, namely those sets in $||R_{[i]}||_M$, naturally represents the possible actions  $i$ can take at the historical stage represented by $M$. (D2) then clearly says that all joint actions are possible. It is also often written as ``whenever $w R_\Box v_i$ for each $i \in \mathsf{Agt}$,  then there is $u$ such that $v_i R_{[i]} u$ for each $i \in \mathsf{Agt}$''. $R^{t}_{\mathsf{Agt}}$ clearly should mean that all people performed the same action at stage $t$ across two worlds, and (D3) precisely says this.  (D4) is the famous `no choice between undivided histories' condition.  It is equivalent to that $R^{t+1}_{\Box} \subseteq R^t_{\mathsf{Agt}}$:  if two worlds are historically identical at stage $t+1$, then everyone must have performed the same action at stage $t$ in those two worlds.

Now (D5) to (D8) are motivated by the following picture: at each stage $t$ and each
equivalence class $M$ of $R^t_\Box$, for each $i \in \mathsf{Agt}$,
among the possible actions $i$ can perform,
some (and at least one) are labeled as allowed,
and the agent ought to perform one of the allowed actions.
(D5) to (D8) then essentially says that for any $h \in M$,
$R^t_{\mathsf{O}_i}[h]$ is the union of the equivalence classes
representing these actions.
Since all the allowed actions are chosen from those that $i$ can perform,
we have (D5).
Since there is at least one action that is allowed, we have (D6).
Since `ought' for $i$ is predicated on her actions and not something
she cannot control, and mathematically $R^t_{\mathsf{O}_i}[h]$ is a union
of the equivalence classes of $R^t_{[i]}$, we have (D7).
Finally, since we take it that given an equivalence class $M$ of $R^t_{\Box}$,
which actions are allowed for $i$ is already determined, we have (D8).
To see that this is a reasonable assumption, note that at stage $t$, agent $i$ has no
resources to distinguish different worlds in an equivalence class $M$ of $R^t_{\Box}$ 
as they are historically identical, and the present and the future are yet to be determined.
Thus, it is unreasonable to give $i$ different duties at stage $t$
at different worlds in $M$.

Now we define the expected semantics for $\mathcal{L}_{\mathrm{DTDS}}$ on interpreted systems.
\begin{definition}
Let $\mathcal{I}=(H, R_\Box, \{R_{[i]}, 
R_{\mathsf{O}_i} \mid i \in Agt\}, R_{\mathsf{Agt}}, V)$ 
be an interpreted system. 
The satisfaction of a formula $\varphi$ in $\mathcal{I}$ is defined according to:
\begin{enumerate}
    \item $(\mathcal{I},h,  t) \models p$ iff $(h, t) \in V(p)$
    \item $(\mathcal{I}, h, t) \models \neg \varphi  $ iff $(\mathcal{I}, h, t) \nvDash \varphi$
    \item $(\mathcal{I}, h, t) \models \varphi \land \psi$ iff $(\mathcal{I}, h, t) \models \varphi \text { and }(\mathcal{I}, h, t) \models \psi$
    
    \item $(\mathcal{I},h, t) \models \square \varphi $ iff $ \forall (h^{\prime}, t^{\prime})\in H\times\mathbb{N}:\text{ if }(h, t) R_{\Box}\left(h^{\prime},t^{\prime}\right) , \text{ then }(\mathcal{I}, h', t') \models \varphi $
    \item $(\mathcal{I}, h, t) \models[i] \varphi $ iff $\forall\left(h^{\prime}, t^{\prime}\right) \in H\times\mathbb{N}: \text { if }(h, t) R_{[i]}\left(h^{\prime}, t^{\prime}\right), \text { then }\left(\mathcal{I}, h^{\prime}, t^{\prime}\right) \models \varphi$
    \item $(\mathcal{I},h, t) \models [\mathsf{Agt}] \varphi $ iff $  \forall\left(h^{\prime}, t^{\prime}\right) \in H\times\mathbb{N}:\text{ if }(h, t) R_{\mathsf{Agt}}\left(h^{\prime},t^{\prime}\right) , \text{ then }(\mathcal{I}, h', t') \models \varphi$
    \item $(\mathcal{I}, h, t) \models\mathsf{O}_i \varphi $ iff $  \forall\left(h^{\prime}, t^{\prime}\right) \in H\times\mathbb{N}:\text{ if }(h, t) R_{\mathsf{O}_i}\left(h^{\prime},t^{\prime}\right) , \text{ then }(\mathcal{I}, h', t') \models \varphi$
    \item $(\mathcal{I}, h, t) \models\mathsf{X} \varphi $ iff $ \left(\mathcal{I}, h, t+1\right) \models \varphi$
    \item $(\mathcal{I}, h, t) \models \mathsf{U}(\varphi,\psi)$ iff there is some  $t^{\prime} \ge t$ such that $\left(\mathcal{I}, h, t^{\prime}\right) \models \varphi$ and for all $t^{\prime \prime}$ with $t \le t^{\prime \prime} < t^{\prime}$, $\left(\mathcal{I}, h, t^{\prime \prime}\right) \models \psi$.
\end{enumerate}
We say that $\varphi$ is \emph{valid} if for any interpreted system $\mathcal{I}$, 
any $h$ in the set $H$ of $\mathcal{I}$, and any $t \in \mathbb{N}$, 
$(\mathcal{I}, h, t) \models \varphi$.
\end{definition}
Our goal is to axiomatize the valid formulas.

\subsection{Logic}
The following defines the logic $\mathsf{L}_{\mathrm{DTDS}}$. 
\begin{definition}
    Let $\mathsf{L}_{\mathrm{DTDS}}$ be the smallest subset of $\mathcal{L}_{\mathrm{DTDS}}$
    such that it contains all instances of the following axiom schemata 
    for all $i \in \mathsf{Agt}$:
    \begin{itemize}
        \item {\upshape (PC)} all propositional tautologies 
        \item {\upshape ($\mathsf{S5}_{\Box[i][\mathsf{Agt}]}$)} the $\mathsf{K}$, $\mathsf{T}$, $\mathsf{4}$, and $\mathsf{5}$ axioms for $\Box$, 
        $[i]$, and $[\mathsf{Agt}]$;
        \item {\upshape ($\mathsf{K}_{\mathsf{O}_i\mathsf{X}}$)}
        the $\mathsf{K}$ axioms for 
        $\mathsf{O}_i$ and $\mathsf{X}$
        \item {\upshape (A1)} $\Box\varphi \to [i]\varphi$ 
        \item {\upshape (A2)}
        $\bigwedge_{i \in \mathsf{Agt}}\Diamond[i]\varphi_i \to 
        \Diamond\bigwedge_{i \in \mathsf{Agt}}[i]\varphi_i$ 
        \item {\upshape (A3)}
        $\bigwedge_{i \in \mathsf{Agt}}[i]\varphi_i \to 
        [\mathsf{Agt}]\bigwedge_{i \in \mathsf{Agt}}\varphi_i$ 
        \item {\upshape (A4)} 
        $[\mathsf{Agt}]\mathsf{X}\varphi \to \mathsf{X}\Box\varphi$
        \item {\upshape (A5)} 
        $\Box\varphi \to \mathsf{O}_i\varphi$
        \item {\upshape (A6)} 
        $\mathsf{O}_i \varphi \to \lnot \mathsf{O}_i \lnot \varphi$
        \item {\upshape (A7)} 
        $\mathsf{O}_i \varphi \to \mathsf{O}_i [i] \varphi$ 
        \item {\upshape (A8)} 
        $\mathsf{O}_i \varphi \to \Box \mathsf{O}_i \varphi$ 
        \item {\upshape ($\mathsf{X}$Func)} 
        $\mathsf{X}\varphi \leftrightarrow \lnot \mathsf{X} \lnot \varphi$
        \item {\upshape ($\mathsf{U}$Fix)}
        $\mathsf{U}(\varphi, \psi) \leftrightarrow 
            (\varphi \lor (\psi \land \mathsf{X} \mathsf{U}(\varphi, \psi) ))$
        \end{itemize}
        and is closed under the following rules:
        \begin{itemize}[itemsep=0.4em]
        \item {\upshape ($\mathsf{U}$Ind)}
        $\displaystyle\frac{\chi \to (\lnot \varphi \land \mathsf{X}\chi)}{\chi \to \lnot \mathsf{U}(\varphi, \psi)}$;
        \item {\upshape (MP)} $\displaystyle\frac{\varphi\quad (\varphi \to \psi)}{\psi}$;
        \item {\upshape (Nec)} 
        $\displaystyle\frac{\varphi}{M\varphi}$ 
        for every $M \in \{\Box, [\mathsf{Agt}], \mathsf{X}\} \cup 
        \{[i], \mathsf{O}_i \mid i \in \mathsf{Agt}\}$.
    \end{itemize}
\end{definition}
\noindent The axioms (A1) to (A8), all of which are Sahlqvist formulas,
correspond to the properties (D1) to (D8), except for (A3) and (D3); 
(A3) corresponds to the super-additivity condition (D3*) in Definition \ref{def:super-additivity}.
Their intuitive meaning should be self-explanatory given our exposition of the semantics. 
In particular, (A7) means that $\mathsf{O}_i$, 
which is really saying that `$i$ ought to see to it that'.
Then, the main theorem to which the rest of the paper is devoted is:
\begin{theorem}\label{thm:completeness}
    For any $\varphi \in \mathcal{L}_{\mathrm{DTDS}}$, $\varphi$ is valid iff
    $\varphi \in \mathsf{L}_{\mathrm{DTDS}}$.
\end{theorem}

\section{More general semantics}\label{sec:general-semantics}
We will follow the standard strategy of first building the canonical model
of our axiomatic system $\mathsf{L}_{\mathrm{DTDS}}$, 
and then transform that into an interpreted system
to satisfy some consistent formula $\varphi$. 
This process goes through two intermediate kinds of models which we call
`super-additive interpreted systems' and `Kripke premodels'. 
In the following two subsections, 
we define them and and also the required transformations. 

\subsection{Super-additive interpreted systems}
We first define super-additive interpreted systems. 
They are just slight generalization of the interpreted systems, 
where only $R_{\mathsf{Agt}} \subseteq \bigcap_{i \in \mathsf{Agt}}R_{[i]}$
is required. 
\begin{definition}\label{def:super-additivity}
    A super-additive interpreted system is a tuple 
    $\mathcal{I}=(H, R_\Box, 
    \{R_{[i]},  R_{\mathsf{O}_i} \mid i \in \mathsf{Agt}\}, 
    R_{\mathsf{Agt}}, V)$ 
    satisfying all the requirements in Definition \ref{def:interpreted-system} 
    except that {\upshape (D3)} is replaced by the following:
    \begin{itemize}
    \item[]
    \begin{itemize}
        \item[\upshape (D3*)] $R_{\mathsf{Agt}} \subseteq \bigcap_{i \in \mathsf{Agt}}R_{[i]}$.
    \end{itemize}
    \end{itemize}
\end{definition}
Now we show that any super-additive interpreted system can be 
transformed into an interpreted system satisfying the same formulas. 
In fact, every super-additive interpreted system is the p-morphic 
image of some interpreted system. 
We first need a lemma for this.
\begin{lemma}\label{lem:choice-expansion}
    For any non-empty sets $X$ and $I$, 
    there is a function $F$ from $X^I$ to $X$ 
    such that for any $i \in I$ and $x \in X$, 
    \[
        \{F(f_{-i} \cup \{(i, x)\}) \mid f_{-i} \in X^{I \setminus \{i\}}\} = X.
    \]
\end{lemma}
\begin{proof}
    When $|I| \ge 3$, we can prove this easily by first fix an $x_0 \in X$ and 
    define $F$ by `super-majority vote': for any $f \in X^I$, think of $f(i)$ 
    as $i$'s vote, and if there is $x \in X$ such that all but at most one $i \in I$ 
    votes for $x$, then let $F(f) = x$; otherwise let $F(f) = x_0$. 
    It is easy to see that this $F$ satisfies the requirement. 

    When $|I| = 2$, let us assume without loss of generality that $X$ is an 
    initial ordinal $\kappa$,\footnote{When $X$ is infinite, 
    this uses the axiom of choice.}
    and also let us assume that $I = \{0, 1\}$ and write any 
    $f \in X^I$ as a pair $(x, y)$ where $x = f(0)$ and $y = f(1)$. 
    If $X$ is finite, then $X$ is simply some $\{0, 1, 2, \dots, m\}$. 
    In this case, define $F((x, y)) = (x + y) \mod (m+1)$, 
    and it is easy to see that this works.\footnote{For the completeness proof, we will only apply this lemma to finite models, and thus the rest of the proof is not needed.}
    If $X$ is infinite, then recall that every ordinal $\gamma$ can be uniquely 
    written as $\omega \cdot \alpha + m$ where $\alpha$ is an ordinal and 
    $m$ is a natural number, and the parity of $\gamma$ is defined as 
    the parity of $m$. 
    Now define $F((\gamma_1, \gamma_2))$ as follows: first decompose
    $\gamma_1 = \omega \cdot \alpha_1 + m_1$ and 
    $\gamma_2 = \omega \cdot \alpha_2 + m_2$; then, if $m_1 + m_2$ is odd, 
    let $F((\gamma_1, \gamma_2)) = \alpha_1 + \floor{m_1/2}$,
    else let $F((\gamma_1, \gamma_2)) = \alpha_2 + \floor{m_2/2}$.
    When $\kappa$ is the countable initial ordinal $\omega$, 
    $\alpha_1$ and $\alpha_2$ are always $0$, 
    but both $\{\floor{m / 2} \mid m \text{ is odd}\}$ and 
    $\{\floor{m / 2} \mid m \text{ is even}\}$ are $\omega$. 
    When $\kappa$ is uncountable, for both $m = 0, 1$, 
    $\{\alpha \mid \omega \cdot \alpha + m\in \kappa\}$ is just $\kappa$. 
    Then it is easy to see that $F$ satisfies the requirement.
\end{proof}
Now we define p-morphisms bewteen super-additive interpreted systems 
without the valuation parts, which we call super-additive interpreted frames.
\begin{definition}
    A (resp. super-additive) DTDS interpreted \emph{frame} is a 
    (resp. super-additive) DTDS interpreted system without the valuation $V$ part. 

    Let $\mathbf{I} = (H, R_\Box, \{R_{[i]}, R_{\mathsf{O}_i} \mid i \in \mathsf{Agt}\}, R_{\mathsf{Agt}})$
    and 
    $\mathbf{I}' = (H', R'_\Box, \{R'_{[i]}, R'_{\mathsf{O}_i} \mid i \in \mathsf{Agt}\}, R'_{\mathsf{Agt}})$
    be two super-additive interpreted frames. 
    A \emph{p-morphism} $\pi$ is a function from $H$ to $H'$ such that 
    for any $t \in \mathbb{N}$ and for any $O \in \{\Box, [\mathsf{Agt}]\} \cup 
    \{[i], \mathsf{O}_i\mid i \in \mathsf{Agt}\}$,
    \begin{itemize}
        \item for any $g, h \in H$, 
            if $(g, t)R_O(h, t)$ then $(\pi(g), t)R'_O(\pi(h), t)$, and 
        \item for any $g \in H$, if there is $h' \in H'$ such that
        $(\pi(g), t)R'_O(h', t)$, then there is $h \in H$ such that 
        $\pi(h) = h'$ and $(g, t)R_O(h, t)$.
    \end{itemize}
\end{definition}
As is in the standard modal case, p-morphism reflects satisfiability. 
\begin{lemma}\label{lem:frame-to-model}
    Let $\mathbf{I}$ and $\mathbf{I}'$ be two DTDS super-additive interpreted frames. 
    If $\pi$ is a surjective p-morphism from $\mathbf{I}'$ to $\mathbf{I}$, 
    then for any valuation $V$ for $\mathbf{I}$, there is a valuation $V'$ for 
    $\mathbf{I}$ such that for any $(h', t) \in H' \times \mathbb{N}$, 
    $(\mathbf{I}, V), (\pi(h'), t)$ and $(\mathbf{I}', V'), (h', t)$ satisfy
    the same formulas.
\end{lemma}
\begin{proof}
    The valuation $V'$ is constructed in the natural way: $(h', t) \in V'(p)$ 
    iff $(\pi(h'), t) \in V(p)$. The rest is proved by a standard induction.
\end{proof}
Now we show that any super-additive interpreted frame is the p-morphic image
of an interpreted frame.
\begin{lemma}\label{lem:convert-to-additive}
    Let $\mathbf{I}$ be a super-additive DTDS interpreted frame. 
    Then there is a DTDS interpreted frame $\mathbf{I}'$ and a surjective function $\pi$
    from $H$ to $H'$ that is also a p-morphism.
\end{lemma}
\begin{proof}
    For every equivalence class $M$ of $R_\Box$, let us construct the following:
    \begin{itemize}
        \item For any $M \in ||R_{\Box}||$, we have commented before how 
        $||R_{[i]}||_M$ represents the action space that agent $i$ has 
        at moment $M$. For more suggestive notations, 
        let us call this $A^i_{M}$, and let $A_M = \Pi_{i \in \mathsf{Agt}}A^i_m$.
        In light of Lemma \ref{lem:choice-expansion} and what we are trying
        to achieve, we refine $A^i_M$ to 
        $\mathbb{A}^i_{M} = A^i_M \times ||R_{\mathsf{Agt}}||_M$. 
        Let $\mathbb{A}_M$ be the refined joint-action space 
        $\Pi_{i \in \mathsf{Agt}} \mathbb{A}^i_M$ at moment $M$. 
        Note that any $f \in \mathbb{A}_{M}$ can be naturally decomposed 
        into $f_1 \in A_M$ and $f_2 \in (||R_{\mathsf{Agt}}||_M)^\mathsf{Agt}$ where $f(i) = (f_1(i), f_2(i))$. 
        Conversely, use `$+$' to denote the pointwise pair formation that 
        combines $f_1$ and $f_2$ into $f$.
        \item Now apply Lemma \ref{lem:choice-expansion} to 
        $||R_{\mathsf{Agt}}||_M$ and $\mathsf{Agt}$ to obtain $F_M$. 
        For any $a \in A_M$, 
        let $D_a = \bigcap_{i \in \mathsf{Agt}}a(i)$ and also
        fix a cell $c_a \in ||R_{\mathsf{Agt}}||_{D_a}$. 
        Then define function $G_M$ from $\mathbb{A}_M$ to $||R_{\mathsf{Agt}}||_M$
        as follows: for any 
        $f \in \mathbb{A}_M$, if $F_M(f_2) \in D_{f_1}$, 
        then let $G_M(f) = F_M(f_2)$, otherwise let $G_M(f) = c_{f_1}$.
        The extra $c_{f_1}$ makes sure that $G_M(f)$ is always in $D_{f_1}$.
        \item For any history $h \in H$, 
        a \emph{refinement} $\mathfrak{h}$ of it is a function defined 
        on $\mathbb{N}$ such that for each $t \in \mathbb{N}$, 
        letting $M = R_{\Box}[(h, t)]$ (the $R_{\Box}$ equivalence class that
        $(h, t)$ is in), $\mathfrak{h}(t)$ is in $\mathbb{A}_M$ and 
        $G_M(\mathfrak{h}(t))$ is $R_{\mathsf{Agt}}[(h, t)]$. Note that this implies 
        $\mathfrak{h}(t)_1(i) = R_{[i]}[(h, t)]$ for all $i \in \mathsf{Agt}$. 
        \item For each $h \in H$, let $\mathfrak{H}(h)$ be the set of all of
        its refinements. The set $H'$ of histories in the DTDS interpreted frame 
        $\mathbf{I}'$ we construct is the disjoint union of these $\mathfrak{H}(h)$.
        Formally, let 
        $H' = \{(h, \mathfrak{h}) \mid h \in H, \mathfrak{h} \in \mathfrak{H}(h)\}$.
        Naturally, we define the p-morphism $\pi$ by 
        $\pi((h, \mathfrak{h})) = h$.
        \item Now we define the relations in $\mathbf{I}'$. 
        We write $(h, \mathfrak{h}, t)$ for $((h, \mathfrak{h}), t)$. 
        \begin{itemize}
            \item $(h, \mathfrak{h}, t) R'_\Box (h', \mathfrak{h}', t')$ iff 
            $t = t'$, $(h, t) R_\Box (h', t)$, and for all $t_0 < t$, $\mathfrak{h}(t_0) = \mathfrak{h}'(t_0)$.
            \item $(h, \mathfrak{h}, t) R'_{[i]} (h', \mathfrak{h}', t')$ iff 
            $(h, \mathfrak{h}, t)R'_\Box (h', \mathfrak{h}', t')$ and $\mathfrak{h}(t)(i) = \mathfrak{h}'(t)(i)$.
            \item $R'_{\mathsf{Agt}}$ is the intersection of all $R'_{[i]}$.
            \item $(h, \mathfrak{h}, t) R'_{\mathsf{O}_i} (h', \mathfrak{h}', t')$ iff 
            $(h, \mathfrak{h}, t)R'_\Box(h', \mathfrak{h}', t')$ and $(h, t) R_{\mathsf{O}_i} (h', t)$.
        \end{itemize}
    \end{itemize}
    We need to verify that $\mathbf{I}'$ is a DTDS interpreted frame.
    (D0) is trivial by definition. (D1) is similar: if $(h, t)$ and $(h', t)$ 
    are in two different equivalence classes $M$ and $M'$ of $R_\Box$, 
    then $\mathbb{A}^i_M$ and $\mathbb{A}^i_{M'}$ are disjoint, 
    so $\mathfrak{h}(t)(i)$ and $\mathfrak{h}'(t)(i)$ must be different. 
    For (D2), note that for any $(h, \mathfrak{h}, t) \in H' \times \mathbb{N}$, 
    if we have $(h^i, \mathfrak{h}^i, t) R'_\Box (h, \mathfrak{h}, t)$ 
    for all $i \in \mathsf{Agt}$, then all $(h^i, t)$ are in the same moment 
    $M = R_\Box[(h, t)]$, and thus we have a joint refined action $f \in \mathbb{A}_M$
    where $f(i) = \mathfrak{h}^i(t)(i)$. Using $G_M$, we obtain 
    $G_M(f) \in ||R_{\mathsf{Agt}}||_{D_{f_1}}$.
    Pick any $h' \in G_M(f)$. This immediately means that $(h', t) \in M$,
    and thus for any $t' < t$, $(h', t') R_{\mathsf{Agt}} (h, t')$.
    Now construct a refinement $\mathfrak{h}'$ of $h'$ as follows:
    \begin{itemize}
        \item if $t' < t$, let $\mathfrak{h}'(t') = \mathfrak{h}(t')$, 
        \item $\mathfrak{h}'(t) = f$, and 
        \item for $t' > t$, let $M' = R_{\Box}[(h', t')]$ and let 
        $\mathfrak{h}'(t')$ be any joint refined action in $\mathbb{A}_{M'}$
        such that $G_{M'}(\mathfrak{h}'(t'))$ is $R_{\mathsf{Agt}}[(h', t')]$.
    \end{itemize}
    It should be very obvious, using Lemma \ref{lem:choice-expansion}, that 
    $G_{M'}$ from $\mathbb{A}_{M'}$ to $||R_{\mathsf{Agt}}||_{M'}$ is 
    surjective. In fact, Lemma \ref{lem:choice-expansion} ensures surjectivity 
    even if you fix one agent's action. 
    Then, $\mathfrak{h}'$ is a refinement of $h'$, 
    $(h', \mathfrak{h'}, t) R'_{\Box} (h, \mathfrak{h}, t)$, and 
    $(h^i, \mathfrak{h}^i, t) R'_{[i]} (h', \mathfrak{h'}, t)$ 
    for all $i \in \mathsf{Agt}$.

    (D3) is automatic by definition, and (D4) is almost the same, since 
    $R'_{\Box}$ explicitly requires identical joint refined actions in 
    previous stages. The verification of (D5) to (D8) are also easy. 
    
    It remains to verify that $\pi$ is indeed a surjective p-morphism from
    $\mathbf{I}'$ to $\mathbf{I}$. That $\pi$ is surjective is trivial.
    The forward condition for $\pi$ being a p-morphism is also almost trivial. 
    Note that $\mathfrak{h}(t)$ also encodes the information of which 
    unrefined actions the agents take at $t$: 
    if $\mathfrak{h}(t)(i) = \mathfrak{h}'(t)(i)$, then in particular 
    $\mathfrak{h}(t)_1(i) = \mathfrak{h}'(t)_1(i)$, which means 
    $(h, t)$ and $(h', t)$ must be in the same $R_{[i]}$ equivalence class. 

    For the backward conditions, take $[\mathsf{Agt}]$ for example first. 
    Pick any $(h, \mathfrak{h}, t)$ and $(h', t)$ such that 
    $(h, t) R_{\mathsf{Agt}} (h', t)$. 
    Note that this means $R_{\mathsf{Agt}}[(h', t)] = R_{\mathsf{Agt}}[(h, t)]
    = G_M(\mathfrak{h}(t))$ where $M = R_{\Box}[(h, t)]$, since $\mathfrak{h}$
    refines $h$. 
    But then we can easily construct a refinement $\mathfrak{h}'$ of $h'$ such that 
    $(h, \mathfrak{h}, t) R'_{\mathsf{Agt}} (h', \mathfrak{h}', t)$.
    Just let $\mathfrak{h}'$ be identical to $\mathfrak{h}$ for all $t' \le t$, 
    and for $t' > t$, pick any $f \in \mathbb{A}_{M'}$ where $M' = R_\Box[(h', t')]$
    so that $G_{M'}(f) = R_{\mathsf{Agt}}[(h', t')]$. The slightly non-trivial case
    is $[i]$, and here Lemma \ref{lem:choice-expansion} is used. 
    Pick any $(h, \mathfrak{h}, t)$ and $(h', t)$ such that 
    $(h, t) R_{[i]} (h', t)$. Let $M = R_{\Box}[(h, t)]$ and 
    $N = R_{\mathsf{Agt}}[(h', t)]$. Note that $(h', t) \in M$.
    By Lemma \ref{lem:choice-expansion} and our definition of $G_M$, 
    there is a $f \in \mathbb{A}_M$ such that $f(i) = \mathfrak{h}(t)(i)$ 
    and $G_M(f) = N$. 
    Indeed, let $f_1(j) = R_{[j]}[(h', t)]$ for all $j \in \mathsf{Agt}$, 
    and apply Lemma \ref{lem:choice-expansion} to obtain $f_2$ so that 
    $f_2(i) = \mathfrak{h}(t)_2(i)$ yet $F_M(f_2) = N$. 
    Then $f = f_1 + f_2$ is what we need.
    By construction $N \in D_{f_1}$, so $G_M(f) = F_M(f_2) = N$, 
    and $f_1(i) = R_{[i]}[(h', t)] = R_{[i]}[(h, t)]$ since
    $(h, t) R_{[i]} (h', t)$. 
    This means $f_1(i) = \mathfrak{h}(t)_1(i)$ since the later must be 
    $R_{[i]}[(h, t)]$ for $\mathfrak{h}$ to be a refinement of $h$. 
    Then indeed $f(i) = \mathfrak{h}(t)(i)$. With this $f$, define 
    $\mathfrak{h}'$ as before: 
    \begin{itemize}
        \item if $t' < t$, let $\mathfrak{h}'(t') = \mathfrak{h}(t')$, 
        \item $\mathfrak{h}'(t) = f$, and 
        \item for $t' > t$, let $M' = R_{\Box}[(h', t')]$ and let 
        $\mathfrak{h}'(t')$ be any joint refined action in $\mathbb{A}_{M'}$
        such that $G_{M'}(\mathfrak{h}'(t'))$ is $R_{\mathsf{Agt}}[(h', t')]$.
    \end{itemize}
    Then $\mathfrak{h}'$ refines $h'$, and 
    $(h, \mathfrak{h}, t)R'_{[i]}(h', \mathfrak{h}', t)$.
\end{proof}

\subsection{Kripke premodels}
The super-additive interpreted system satisfying 
some given consistent formula $\varphi$ will be obtained 
by choosing appropriate paths from what we call a DTDS Kripke premodel where 
the temporal structure is not explicitly coded by natural numbers,  
but given also by a binary relation $\to$. 
One may think of this as a selective unraveling. 
\begin{definition}
    A \emph{DTDS Kripke premodel} is a tuple 
    $\mathcal{M} = (S, R_\Box, 
    \{R_{[i]}, R_{\mathsf{O}_i}\mid i \in \mathsf{Agt}\},
    R_{\mathsf{Agt}}, 
    \to, V)$ where
    \begin{itemize}
        \item $S$ is a non-empty set,
        \item $R_\Box$ is an equivalence relation on $S$,
        for every $i \in \mathsf{Agt}$, $R_{[i]}$ is an equivalence on $S$ while
        $R_{\mathsf{O}_i}$ is a binary relation on $S$, 
        $\to$ is a serial binary relation on $S$,
        \item $V$ is a function from $\mathsf{Prop}$ to $S$, 
        \item (D1), (D2), (D3*), and (D5) to (D8) hold as well, 
        \item and instead of (D4), (D4*) 
        ${\to} \circ {R_{\Box}} \subseteq {R_{\mathsf{Agt}}} \circ {\to}$ holds.
    \end{itemize}
\end{definition}
It is important to note that we allow a state $s \in S$ to have multiple 
$\to$-predecessors and multiple $\to$-successors.
For future convenience, we define a semantics for $\mathcal{L}_{\mathrm{DTDS}}$ on these
Kripke premodels, even though this is not the intended semantics 
(for example, it does not validate $X\varphi \leftrightarrow \lnot X \lnot \varphi$).
\begin{definition}
    We recursively define satisfaction for formulas $\varphi \in \mathcal{L}_{\mathrm{DTDS}}$
    on DTDS Kripke premodels as follows:
    let $\M = (W,R_\Box,\{R_{[i]}, R_{\mathsf{O}_i} \mid i\in Agt\}, R_{\mathsf{Agt}}, \to, V)$ 
    be a DTDS Kripke premodel, and
    \begin{itemize}
    \item $\M,w\vDash p$ iff $w\in V(p)$;
    \item $\M,w\vDash\neg\varphi$ iff $\M,w \nvDash \varphi$;
    \item $\M,w\vDash\varphi\land\psi$ iff $\M,w \vDash \varphi$ and $\M,w \vDash \psi$;
    \item $\M,w\vDash\Box\varphi$ iff $\forall u\in R_{\Box}[w]$, $\M,u\vDash\varphi$;
    \item $\M,w\vDash[i]\varphi$ iff $\forall u\in R_{[i]}[w]$, $\M,u\vDash\varphi$;
    \item $\M,w\vDash[Agt]\varphi$ iff $\forall u\in R_{\mathsf{Agt}}[w]$, $\M,u\vDash\varphi$;
    \item $\M,w\vDash\mathsf{O}_{i}\varphi$ iff $\forall u\in R_{\mathsf{O}_i}[w]$, $\M,u\vDash\varphi$;
    \item $\M,w\vDash\mathsf{X}\varphi$ iff  $\forall u\in {{\to}[w]},\M,u\vDash\varphi$;
    \item $\M, w \vDash \mathsf{U}(\varphi, \psi)$ iff there is a finite sequence 
    $w = v_0 \to v_1 \to v_2 \to \dots \to v_n$ ($n \ge 0$) such that $\M, v_n \vDash \varphi$
    and for every $i$ from $0$ to $n-1$, $\M, v_i \vDash \psi$.
    \end{itemize}
    A formula $\varphi$ is valid in $\mathcal{M}$ ($\mathcal{M} \models \varphi$)
    if for all $w \in W$, $\mathcal{M}, w \models \varphi$.
\end{definition}
Since $\to$ is not required to be functional on premodels, one may consider another 
natural semantics for $\mathsf{U}$ that universally quantifies over all $\to$-paths
starting from $w$. 
But premodels only serve an instrumental role in this paper, 
so we will not consider all possibilities. 

Not every DTDS Kripke premodel can be transformed into an interpreted system
satisfying the same formulas, since in DTDS Kripke premodels, 
the temporal $\to$ relation is not functional and thus the ($\mathsf{X}$Func)
axiom is often false. 
The axiom ($\mathsf{U}$Fix) is also important in the transformation. 
The ideal situation is when we have a DTDS Kripke premodel on which the full logic 
$\mathsf{L}_{\mathrm{DTDS}}$ is valid, but we will not have one that also satisfies 
a required consistent formula merely from the canonical model of $\mathsf{L}_{\mathrm{DTDS}}$
since a filtration step is involved. 
Thus, we make sure in the following that we only need the validity of relevant 
instances of ($\mathsf{X}$Func) and ($\mathsf{U}$Fix). 
%\begin{definition}
%    We say a set $\Sigma \subseteq \mathcal{L}_{\mathrm{DTDS}}$ is Fischer-Ladner closed if
%    it is closed under subformulas and moreover 
%    if $\mathsf{U}(\varphi, \psi) \in \Sigma$ 
%    then $\mathsf{X}\mathsf{U}(\varphi, \psi) \in \Sigma$.
%
%    For any $\Sigma \subseteq \mathcal{L}_{\mathrm{DTDS}}$, let $\Sigma^b$ be the closure 
%    of $\Sigma$ under Boolean connectives. 
%\end{definition}
%Now we show that if $\Sigma$ is a finite Fischer-Ladner closed set of formulas, 
%and a DTDS Kripke premodel validates $\mathsf{L}_{\mathrm{DTDS}} \cap \Sigma^b$, 
%then we can construct a super-additive interpreted system satisfying the same 
%formulas in $\Sigma$.
\begin{lemma}\label{lem:selective-unravel}
    Let $\Sigma \subseteq \mathcal{L}_{\mathrm{DTDS}}$ be finite and closed under subformulas, 
    and let 
    $\mathcal{M} = 
        (S, R_\Box, \{R_{[i]}, R_{\mathsf{O}_i}\mid i \in \mathsf{Agt}\}, R_{\mathsf{Agt}}, \to, V)$ 
    be a DTDS Kripke premodel such that
    \begin{itemize}
        \item for any $\mathsf{X}\varphi \in \Sigma$,
            $\mathcal{M} \models \mathsf{X}\varphi \leftrightarrow \lnot\mathsf{X}\lnot\varphi$, and
        \item for any $\mathsf{U}(\alpha, \beta) \in \Sigma$, 
            $\mathcal{M} \models \mathsf{U}(\alpha, \beta) \leftrightarrow 
                (\alpha \lor (\beta \land \mathsf{X}\mathsf{U}(\alpha, \beta)))$.
    \end{itemize}
    Then there is a DTDS interpreted system 
    $\mathcal{I} = (H, R^I_\Box, \{R^I_{[i]}, R^I_{\mathsf{O}_i} \mid i \in \mathsf{Agt}\}, 
    R^I_{\mathsf{Agt}}, V')$ and a surjective function $\pi$ form $H \times \mathbb{N}$ to $S$ 
    such that for any $\varphi \in \Sigma$ and $(h, t) \in H \times \mathbb{N}$, 
    $\mathcal{I}, (h, t) \models \varphi$ iff  $\mathcal{M}, \pi(h, t) \models \varphi$.
\end{lemma}
\begin{proof}
    We call a function $h$ from $\mathbb{N}$ to $S$ an \emph{acceptable path} 
    if 
    \begin{itemize}
        \item for any $t \in \mathbb{N}$, $h(t) \to h(t+1)$, and 
        \item for any $t \in \mathbb{N}$ and any $\mathsf{U}(\alpha, \beta) \in \Sigma$, 
        if $\mathcal{M}, h(t) \models \mathsf{U}(\alpha, \beta)$, 
        then there is $t' \ge t$ such that $\M, h(t') \models \alpha$ and for all $t''$
        from $t$ to $t' - 1$, $\M, h(t'') \models \beta$.
    \end{itemize}
    $\mathcal{I}$ is defined as follows:
    \begin{itemize}
        \item $H$ is the set of all acceptable paths.
        \item $(h, t) R^I_\Box (h', t')$ iff $t = t'$, $h(t) R_\Box h'(t)$, and for all 
        $t'' < t$, $h(t'') R_{\mathsf{Agt}} h'(t'')$.
        \item For any $O \in \{\mathsf{Agt}\} \cup 
                        \{[i], \mathsf{O}_i \mid i \in \mathsf{Agt}\}$,
            $(h, t) R^I_O (h', t')$ iff $(h, t) R^I_\Box (h', t')$ and $h(t) R_O h'(t)$.
        \item $(h, t) \in V^I(p)$ iff $h(t) \in V(p)$.
    \end{itemize}
    Now we show that $\mathcal{I}$ is what we want. 

    First of all, we show that every finite path along $\to$ can be extended 
    to an acceptable path. This is a standard argument in temporal logic.
    $f$ be a finite path along $\to$, namely a function defined on 
    some natural number $m > 0$ 
    such that if $t, t+1 < m$, then $f(t) \to f(t+1)$. 
    Now iteratively extend $f$ by finding the first defect 
    and fix it. There are two kinds of defects: 
    \begin{itemize}
        \item For some $t < m$, there is an `until' formula 
        $\mathsf{U}(\alpha, \beta) \in \Sigma$ such that 
        $\mathcal{M}, f(t) \models \mathsf{U}(\alpha, \beta)$ yet there is no $t'$
        with $t \le t' < m$ such that $\mathcal{M}, f(t') \models \alpha$ and 
        for all $t''$ from $t$ to $t' - 1$, $\mathcal{M}, f(t'') \models \beta$.
        \item $f(m-1)$ has no successor in the path $f$.
        \item Defects are ordered so that 
        all of $t$'s defects are ordered before all of $t+1$'s defects. 
        $t$'s defects internally can be ordered arbitrarily as there are only finitely 
        many of them.
    \end{itemize}
    If the first defect is of the first type with $t$ and formula 
    $\mathsf{U}(\alpha, \beta)$, we construct the extension $f'$ as follows:
    \begin{itemize}
        \item If $f(t)$ is the end of sequence $f$, then by the semantics of 
            DTDS Kripke premodel, we can easily extend $f$ to $f'$ (non-trivially since this is a defect)
            and fix this defect. 
        \item If $f(t)$ is not the last element of $f$ (i.e. if $t < m-1$),
            the we use the assumption that 
            $\mathcal{M} \models \mathsf{U}(\alpha, \beta) \leftrightarrow 
            (\alpha \lor (\beta \land \mathsf{X}\mathsf{U}(\alpha, \beta)))$.
            By induction it is clear that for all $t'$ with $t \le t' < m$, 
            $\mathcal{M}, f(t) \models \beta$, and 
            $\mathcal{M}, f(m-1) \models \mathsf{X}\mathsf{U}(\alpha, \beta)$.
            By the semantics of DTDS Kripke premodel and the requirement that $\to$ is serial, 
            we can easily extend $f$ non-trivially to fix the defect $\mathsf{U}(\alpha, \beta)$ 
            at $t$.
    \end{itemize}
    If the first defect is of the second type, then we can simply use the seriality of $\to$
    to extend $f$ by one step along $\to$ and obtain $f'$. We can then fix the first defect 
    of $f'$ to obtain $f''$ and continue this process. The union of $f, f', f'', \dots$ 
    will be an acceptable path as it has no defects.

    Now we show that $\mathcal{I}$ is a super-additive interpreted system. 
    The only condition that need some special attention is (D2), namely 
    every joint action is executable. 
    Suppose we have some $t \in \mathbb{N}$ and $h$ and $h^i$ for each $i \in \mathsf{Agt}$ 
    all in $H$ so that $(h, t) R^I_\Box (h^i, t)$ for all $i \in \mathsf{Agt}$. 
    Since $\M$ is a DTDS Kripke premodel, it satisfy its own (D2). 
    Thus we obtain a state $s \in R_\Box[h(t)]$ such that 
    for any $i \in \mathsf{Agt}$, $h^i(t) R_{[i]} s$. 
    Now we need to build an acceptable path $h'$ around $s$. 
    First we extend $s$ to the past to obtain a finite sequence $f$ of length $t + 1$
    so that $f(t) = s$ and for any $t' < t$, $f(t') R_{\mathsf{Agt}} h(t')$. 
    This can be done by repeatedly using (D4*) on $\mathcal{M}$. 
    Suppose $f(t'+1)$ has been defined so that $f(t'+1) R_\Box h(t'+1)$, 
    since $h(t') \to h(t' + 1)$, by (D4*) we obtain a $w \in S$ 
    such that $w R_{\mathsf{Agt}} h(t')$ and $w \to f(t'+1)$. 
    Let $f(t')$ be this $w$. Since $R_\mathsf{Agt} \subseteq R_\Box$, 
    if $t'$ is still greater than $0$, we can continue this process and find $f(t'-1)$.
    Then $f$ can be extended to an acceptable history $h' \in H$. 
    By construction $(h', t) R^I_{\Box} (h, t)$ and also every $(h^i, t)$, and $h'(t) = s$.
    Thus $(h', t) R^I_{[i]} (h^i, t)$ for all $i \in \mathsf{Agt}$, and we verified (D2).

    Finally, by an induction on formulas in $\Sigma$, we can show that 
    for any $\varphi \in \Sigma$, $\mathcal{I}, (h, t) \models \varphi$ iff
    $\mathcal{M}, h(t) \models \varphi$.
    Then only interesting parts are the induction steps for $\mathsf{X}$ and 
    $\mathsf{U}$. 
    \begin{itemize}
        \item For the $\mathsf{X}$ case, if $\mathcal{M}, h(t) \models \mathsf{X}\varphi$, 
        then $\mathcal{I}, (h, t) \models \mathsf{X}\varphi$ as well 
        since $\mathcal{M}, h(t+1) \models \varphi$ by semantics of $\mathsf{X}$ on $\mathcal{M}$ 
        and by induction hypothesis $\mathcal{I}, (h, t+1) \models \varphi$. 
        Conversely, suppose $\mathcal{M}, h(t) \not\models \mathsf{X}\varphi$. 
        We have assumed that 
        $\mathcal{M} \models \mathsf{X}\varphi \leftrightarrow \lnot\mathsf{X}\lnot\varphi$.
        Thus $\mathcal{M}, h(t) \models \mathsf{X} \lnot\varphi$, and by semantics 
        $\mathcal{M}, h(t + 1) \not\models \varphi$. 
        By inductive hypothesis, $\mathcal{I}, (h, t+1) \not\models \varphi$, 
        and by semantics, $\mathcal{I}, (h, t) \not\models \mathsf{X}\varphi$.
        \item For the $\mathsf{U}$ case, if $\mathcal{M}, h(t) \models \mathsf{U}(\alpha, \beta)$, 
        then recall that $h$ is an acceptable path, and by definition and inductive hypothesis 
        on $\alpha$ and $\beta$, $\mathcal{I}, (h, t) \models \mathsf{U}(\alpha, \beta)$. 
        Conversely, if $\mathcal{I}, (h, t) \models \mathsf{U}(\alpha, \beta)$, 
        then the semantics on interpreted system provides means that for some $t' \ge t$, 
        the sequence $h(t) \to h(t+1) \to \dots \to h(t')$ witness 
        $\mathcal{M}, h(t) \models \mathsf{U}(\alpha, \beta)$ 
        (again assuming inductive hypothesis on $\alpha$ and $\beta$).
    \end{itemize}
    This means we can take $\pi$ to be the function that map $(h, t)$ to $h(t)$. 
    It is surjective since for any $s \in S$, it can be extended to an acceptable path, 
    meaning that there is $h \in H$ such that $\pi((h, 0)) = h(0) = s$.
    This concludes the proof.
\end{proof}

\section{Canonical model, filtration, and the completeness proof}\label{sec:completeness}
In this section we complete the proof of completeness by taking a suitable filtration
of the canonical model of $\mathsf{L}_{\mathrm{DTDS}}$, and note that it satisfies the 
requirement of Lemma \ref{lem:selective-unravel}. 

\subsection{Canonical model}
\begin{definition}
    Let $\mathcal{M}^c$ be the canonical model of $\mathsf{L}_{\mathrm{DTDS}}$. 
    That is, 
    $\mathcal{M}^c = 
        (S^c, R^c_\Box, \{R^c_{[i]}, R^c_{\mathsf{O}_i} \mid \in \mathsf{Agt}\}, 
        R^c_{\mathsf{Agt}}, \to^c, V^c)$
    where 
    \begin{itemize}
        \item $S^c$ is the set of all $\mathsf{L}_{\mathrm{DTDS}}$-maximally consistent sets (MCSs);
        for later notational convenience, we use lower case letters `u', `v', `w', `x', `y', and `z'
        to denote MCSs;
        \item for any $O \in \{\Box, \mathsf{Agt}\} \cup \{[i], \mathsf{O}_i \mid i \in \mathsf{Agt}\}$,
        $w R^c_O v$ iff whenever $O\varphi \in w$, $\varphi \in v$;
        \item $w \to^c v$ iff whenever $\mathsf{X}\varphi \in w$, $\varphi \in v$;
        \item $w \in V^c(p)$ iff $p \in w$.
    \end{itemize}
\end{definition}
\noindent The following follows from standard general results for normal propositional modal logic
and Sahlqvist correspondence theory. 
\begin{lemma}
    If $\varphi$ is consistent, then $\varphi$ is in an MCS.
\end{lemma}
\begin{lemma}
    For any non-temporal operator $O$ and any MCS $w$, if $\lnot O\varphi \in w$
    then there is an MCS $v$ such that 
    $w R^c_O v$ and $\lnot\varphi \in v$. 
    Similarly, if $\lnot \mathsf{X} \varphi \in w$, then there is $v$
    such that $w \to^c v$ and $\lnot\varphi \in v$. 
\end{lemma}
\begin{lemma}
    $\mathcal{M}^c$ is a DTDS Kripke premodel and moreover $\to^c$ is a function.
\end{lemma}
The problem with $\mathcal{M}^c$ is that if we use the semantics for 
DTDS Kripke premodel, then the standard truth lemma fails. 
There can be an MCS $w$ containing $\mathsf{U}(p, \top)$, but 
every $v$ reachable from $w$ by $\to^c$ does not have $p$. 
To fix this, we take a filtration.

\subsection{Filtration}
We want to filtrate $\mathcal{M}^c$ through a finite set $\Sigma$ of formulas 
to a DTDS Kripke premodel $\mathcal{M}^f$ where we can prove a truth lemma. 
For this to work, $\Sigma$ must have additional closure properties other than
closure under subformulas, and the filtration itself must be done more finely.
\begin{definition}
    A set $\Sigma \subseteq \mathcal{L}_{\mathrm{DTDS}}$ is \emph{filtration-ready} if 
    \begin{itemize}
        \item $\Sigma$ is closed under subformulas;
        \item whenever $\varphi \in \Sigma$, $\dot\lnot\varphi \in \Sigma$, 
            where $\dot\lnot$ is the no-redundancy negation:
            $\dot\lnot\lnot\psi = \psi$ and $\dot\lnot\psi = \lnot\psi$ 
            if $\psi$ does not start with negation;
        \item whenever $\mathsf{U}(\alpha, \beta) \in \Sigma$ 
        then $\mathsf{X}\mathsf{U}(\alpha, \beta) \in \Sigma$;
        \item whenever $\mathsf{O}_i\varphi \in \Sigma$
        then $[i]\varphi \in \Sigma$.
    \end{itemize}
\end{definition}
For the rest of this section, let us fix a finite filtration-ready $\Sigma$. 
Now we define the filtration of $\mathcal{M}^c$ through $\Sigma$. 
\begin{definition}\label{def:filtration}
    Let $\mathcal{M}^f$, the filtration of the canonical model $\mathcal{M}^c$ 
    through $\Sigma$, be $(S^f, R^f_\Box, \{R^f_{[i]}, R^f_{\mathsf{O}_i} \mid \in \mathsf{Agt}\}, 
        R^f_{\mathsf{Agt}}, \to^f, V^f)$, defined as follows.
    \begin{itemize}
        \item First, for any MCS $w \in S^c$, let $\Sigma(w) = w \cap \Sigma$. 
        Then, define an equivalence relation $\sim$ on $S^c$ by 
        \[w \sim v \Leftrightarrow \Sigma(w) = \Sigma(v) \text{ and } 
            \{\Sigma(x) \mid w R^c_\Box x\} = \{\Sigma(x) \mid v R^c_\Box x\}.\]
        For any $w \in S^c$, let $|w|$ be the $\sim$-equivalence class $w$ is in. 
        $S^f$ is defined as the set of equivalence classes of $\sim$ 
        ($||{\sim}||$ in our notation).
        \item For any non-temporal operator $O$, let $R^e_O$ be the 
        `existential lifting' of $R^c_O$ to $S^f$ defined by $C R^e_O D$ iff
        there is $w \in C$ and $v \in D$ such that $w R^c_O v$.
        \item Let $R^f_\Box$ simply by $R^e_\Box$.
        \item Let $\to^f$ also be the `existential lifting' of $\to^c$. 
            That is, $C \to^f D$ iff there is $w \in C$, $v \in D$ such that $w \to^c v$. 
        \item Let $R^f_{\mathsf{Agt}}$ be the transitive closure of $R^e_{\mathsf{Agt}}$. 
            Also, for any $i \in \mathsf{Agt}$, let $R^f_{[i]}$ be the transitive closure of 
            $R^e_{[i]}$.
        \item For any $i \in \mathsf{Agt}$
            let $R^f_{\mathsf{O}_i}$ be $R^e_{\mathsf{O}_i} \circ R^f_{[i]}$
        \item For any $p \in \mathsf{Prop} \cap \Sigma$, $V^f(p) = \{|w| \mid p \in w\}$.
        $V^f(p) = \varnothing$ for other $p$ not in $\Sigma$. 
    \end{itemize}
\end{definition}

Now we start to verify that $\mathcal{M}^f$ is a DTDS Kripke premodel and satisfies 
the further requirements for Lemma \ref{lem:selective-unravel}.
\begin{lemma}\label{lem:commute}
   ${\sim} \circ {R^c_\Box} = {R^c_\Box} \circ {\sim}$. 
\end{lemma}
\begin{proof}
    It is enough to show only one direction, since the other follows by 
    taking the converse and using the fact that both $\sim$ and $R^c_\Box$ are symmetric. 
    Suppose $x \sim y R^c_\Box z$. 
    By the second clause of the definition of $\sim$, there exists a $z'$ such that  
    $xR^c_\Box z'$ and $\Sigma(z')=\Sigma(y)$. We claim that $z'\sim z$. 
    Since $R^c_\Box$ is an equivalence relation,  
    $\{w \mid xR^c_\Box w\} = \{w \mid z'R^c_\Box w\}$ and 
    $\{w \mid zR^c_\Box w\} = \{w \mid yR^c_\Box w\}$. 
    From $x\sim y$, we infer that $\{\Sigma(w)\mid xR_\Box w\}=\{\Sigma(w)\mid yR_\Box w\}$. 
    Therefore, $\{\Sigma(w) \mid zR_\Box w\} = \{\Sigma(w) \mid z'R_\Box w\}$. 
    Then by the definition of $\sim$, $z'\sim z$.
\end{proof}

\begin{proposition}\label{1}
    $R^f_\Box$ is an equivalence relation. 
\end{proposition}
\begin{proof}
   The reflexivity and symmetry of $R^f_\Box$ are easy from the corresponding properties of $R^c_\Box$. 
   To show that $R^f_\Box$ is transitive, suppose $C_1 R^f_\Box C_2 R^f_\Box C_3$. 
   This means there are $w \in C_1$, $u_1, u_2 \in C_2$, and $v \in C_3$ such that 
   $w R^c_\Box u_1$ and $u_2 R^c_\Box v$.
   So $u_1 \sim u_2 R^c_\Box v$.
   Using Lemma \ref{lem:commute},
   there is $v' \in C_3$ such that $u_1 R^c_\Box v'$. 
   Then $w R^c_\Box u_1 R^c_\Box v'$, and since $R^c_\Box$ is transitive, 
   $w R^c_\Box v'$. Thus $C_1 R^f_\Box C_3$.
\end{proof}
\begin{proposition}
    For any STIT operator $O$, $R^f_O$ is an equivalence relation.
\end{proposition}
\begin{proof}
    It is clear that $R^e_O$ is reflexive and symmetric. 
    Since $R^f_O$ is the transitive closure of $R^e_O$, $R^f_O$ is 
    an equivalence relation.
\end{proof}
\begin{proposition}\label{22}
    $\mathcal{M}^f$ satisfies (D1), (D2), and (D3*).
\end{proposition}
\begin{proof}
    For (D1), note that clearly $R^e_{[i]} \subseteq R^f_\Box$. 
    Since the later is an equivalence relation and in particular is transitive, 
    $R^f_{[i]}$, being the transitive closure of $R^e_{[i]}$,
    is a subset of $R^f_{\Box}$.

    For (D2), suppose we have $C R^f_\Box D_i$ for all $i \in \mathsf{Agt}$. 
    Using \ref{lem:commute}, there is $w \in C$ and $v_i \in D_i$ for each 
    $i \in \mathsf{Agt}$ such that $w R^c_\Box v_i$. 
    Now use (D2) for $\mathcal{M}^c$ and obtain a $u$ such that $w R^c_\Box u$ 
    and $u R^c_{[i]} v_i$ for each $i \in \mathsf{Agt}$. 
    Then $C R^f_\Box |u|$ and $|u| R^f_{[i]} D_i$ for each $i \in \mathsf{Agt}$.

    For (D3*), note that since $R^c_{\mathsf{Agt}} \subseteq R^c_{[i]}$ 
    for any $i \in \mathsf{Agt}$ by (D3*) for $\mathcal{M}^c$, 
    $R^f_{\mathsf{Agt}} \subseteq R^f_{[i]}$ too, because 
    existential lifting and taking transitive closure are clearly 
    monotonic operations.
\end{proof}
\begin{proposition}
    $\mathcal{M}^f$ satisfy (D4*).
\end{proposition}
\begin{proof}
    Suppose $C_1 \to^f C_2 R^f_\Box C_3$. 
    Using Lemma \ref{lem:commute}, 
    there are $u \in C_1$, $v \in C_2$, and $w \in C_3$ such that 
    $u \to^c v R^c_\Box w$. By (D4*) for $\mathcal{M}^c$, 
    there is $v'$ such that $u R^c_{\mathsf{Agt}} v' \to^c w$. 
    Then $C_1 R^f_{\mathsf{Agt}} |v'| \to^f C_3$.
\end{proof}
\begin{proposition}\label{2}
    $\mathcal{M}^f$ satisfies (D5) to (D8).
\end{proposition}
\begin{proof}
    We only comment on the key tricks needed to verify them. 
    For (D5), since $R^f_\Box$ is an equivalence class and 
    both $R^e_{\mathsf{O}_i}$ and $R^f_{[i]}$ are subsets of it, 
    their composition must be as well. 
    For (D6), note that the composition of two serial relations 
    must be serial as well. 
    For (D7), note that composition is associative and 
    $R^f_{[i]} \circ R^f_{[i]} = R^f_{[i]}$.
    For (D8), first show 
    $R^c_{\Box} \circ R^e_{\mathsf{O}_i} \subseteq R^e_{\mathsf{O}_i}$ 
    using Lemma \ref{lem:commute}. 
\end{proof}
\begin{lemma}\label{lem:filtration-right-structure}
    $\mathcal{M}^f$ is a DTDS Kripke premodel.
\end{lemma}
\begin{proof}
    Combine the previous propositions.
\end{proof}

Now that we have verified that $\mathcal{M}^f$ is a DTDS Kripke premodel, 
it remains to show the truth lemma for it, 
using the semantics we defined for DTDS Kripke premodels. 
The inductive step for $\mathsf{U}$ requires a lengthier treatment, 
so we do it separately first. 
\begin{lemma}\label{lem:extra-rule}
    $\mathsf{L}_{\mathrm{DTDS}}$ is closed under the following rule:
    \[
        \frac{\varphi \to \neg \alpha \qquad \varphi \to \mathsf{X}(\varphi \vee(\neg \alpha \wedge \neg \beta))}{\varphi \to \neg \mathsf{U}(\alpha, \beta)}.
    \]
\end{lemma}
\begin{proof}
    We use $\vdash$ for provability in $\mathsf{L}_{\mathrm{DTDS}}$.
    Suppose that 
    $\vdash \varphi \to \neg\alpha$ and that
    $\vdash \varphi \to \mathsf{X}(\varphi \vee(\neg \alpha \wedge \neg \beta))$. 
    By the left-to-right ($\mathsf{U}$Fix), 
    $\vdash\mathsf{U}(\alpha,\beta)\to(\alpha\vee(\beta \wedge \mathsf{XU}(\alpha,\beta)))$. 
    So
    $\vdash (\varphi \wedge\mathsf{U}(\alpha, \beta)) \to (\neg \alpha \wedge \mathsf{X}\mathsf{U}(\alpha, \beta) \wedge \mathsf{X}(\varphi \vee(\neg \alpha \wedge \neg \beta)))$. 
    Then, by the left-to-right ($\mathsf{U}$Fix) again, 
    $\vdash\mathsf{U}( \alpha,  \beta) \to ( \alpha \vee \beta)$. 
    Hence, $\vdash\mathsf{X}\mathsf{U}( \alpha,  \beta) \to \mathsf{X}( \alpha \vee  \beta)$ 
    by normal modal reasoning with $\mathsf{X}$. 
    Also, it is easy to notice that 
    $\vdash\mathsf{X}(\varphi\vee(\neg\alpha\wedge\neg\beta))\leftrightarrow\mathsf{X}((\alpha\vee\beta)\to\varphi)$.
    Thus, 
    $\vdash(\mathsf{X}\mathsf{U}(\alpha, \beta) \wedge \mathsf{X}(\varphi \vee(\neg \alpha \wedge \neg \beta))) \to \mathsf{X}\varphi$.
    Therefore, $\vdash(\varphi\wedge\mathsf{U}(\alpha, \beta)) \to(\neg\alpha\wedge\mathsf{X}\varphi)$ 
    and thus 
    $\vdash (\varphi \wedge\mathsf{U}(\alpha, \beta)) \to (\neg \alpha \wedge \mathsf{X}(\varphi \wedge\mathsf{U}(\alpha , \beta)))$, using the left-to-right ($\mathsf{U}$Fix) again.
    Now using ($\mathsf{U}$Ind) we obtain 
    $\vdash (\varphi \wedge\mathsf{U}(\psi, \chi)) \to \neg\mathsf{U}(\psi, \chi) $,
    which implies that $\vdash \varphi \to \neg\mathsf{U}(\psi , \chi)$.
\end{proof}
\begin{definition}
    For any $C \in S^f$, let $\Sigma(C) = \bigcap_{w \in C} \Sigma(w)$. 
    Note that since $C$ is an equivalence class of $\sim$ 
    and if $w \sim v$ then at least $\Sigma(w) = \Sigma(v)$, 
    for any $w \in C$, $\Sigma(C) = \Sigma(w)$ and thus $\Sigma(C)$ is also 
    $\bigcup_{w \in C} \Sigma(w)$.
    Then, for any $C \in S^f$, define $\chi_C$ as the formula
    \[
        \bigwedge \Sigma(C) \land \bigwedge\{\Diamond\bigwedge\Sigma(D) \mid C R^c_\Box D\} 
        \land \Box\bigvee\{\bigwedge\Sigma(D) \mid C R^c_\Box D\}.
    \]
    %For any $\mathcal{C} \subseteq S^f$, let 
    %$\Sigma(\mathcal{C}) = \bigcap_{C \in \mathcal{C}} \Sigma(C)$.
\end{definition}

\begin{lemma}\label{lem:until-truth-lemma}
    For any `until' formula $\mathsf{U}(\alpha, \beta) \in \Sigma$
    and for any $C \in S^f$, 
    $\mathsf{U}(\alpha, \beta) \in \Sigma(C)$ iff 
    there is a sequence $C_0 = C \to^f C_1 \to^f \dots \to^f C_m$ ($m \ge 0$) 
    in $\mathcal{M}^f$ 
    such that for all $k \in \{0, 1,\dots,m-1\}$, $\beta \in \Sigma(C_k)$ and $\alpha \in \Sigma(C_m)$.
\end{lemma}
\begin{proof}
    Note that $\alpha, \beta, \mathsf{X}\mathsf{U}(\alpha, \beta)$ are also in $\Sigma$. 
    For the right-to-left direction, 
    we start by noting that $\mathsf{U}(\alpha, \beta) \in \Sigma(C_m)$ 
    and then induct backwards to show 
    that every $\Sigma(C_k)$ has $\mathsf{U}(\alpha, \beta)$. 
    Since $\alpha \in \Sigma(C_m)$, for any $w \in C_m$, $\alpha \in w$. 
    Since $w$ is an MCS, by ($\mathsf{U}$Fix), $\mathsf{U}(\alpha, \beta) \in w$
    and thus is in $\Sigma(C_m)$. 
    Now suppose for some $k$ with $1 < k \le m$, $\mathsf{U}(\alpha, \beta) \in \Sigma(C_k)$; 
    let us consider $C_{k-1}$. 
    Since $C_{k-1} \to^f C_{k}$, there is $w \in C_{k-1}$ and $v \in C_k$ such that $w \to^c v$. 
    Since $\mathsf{U}(\alpha, \beta) \in \Sigma(C^k)$, $\mathsf{U}(\alpha, \beta) \in v$. 
    By the definition of $\to^c$ and axiom ($\mathsf{X}$Func), $\mathsf{X}\mathsf{U}(\alpha, \beta) \in w$. 
    Since $\beta \in \Sigma(C_k)$, $\beta \in w$. 
    By ($\mathsf{U}$Fix), $\mathsf{U}(\alpha, \beta) \in w$, and thus also in $\Sigma(C_k)$.
    Due to the induction, $\mathsf{U}(\alpha, \beta) \in C_0 = C$.

    For the left-to-right direction, note first that using ($\mathsf{U}$Fix), 
    either $\alpha$ or $\beta$ is in $\Sigma(C)$. 
    If $\alpha \in \Sigma(C)$, then $C$ itself is a required sequence. 
    So suppose $\beta \in \Sigma(C)$. Now define $\mathcal{D}$ as the set of 
    equivalence classes reachable from $C$ in the graph 
    $(\{D \in S^f \mid \beta \in \Sigma(D)\}, \to^f)$. 
    It is easy to see that if either of the following is true, 
    then there is a required sequence:
    \begin{itemize}
        \item[(1)] there is any $D \in \mathcal{D}$ such that 
            $\alpha \in \Sigma(D)$;
        \item[(2)] there is any $D \in \mathcal{D}$ and $D' \in S^f$ such that $D \to^f D'$ 
    and $\alpha \in \Sigma(D')$.
    \end{itemize}
    Now suppose neither of the above is true, and we derive a contradiction. 
    Let $\varphi = \bigvee\{\chi_D \mid D \in \mathcal{D}\}$. 
    Note that for any $w \in S^c$, $|w| \in \mathcal{D}$ iff $\varphi \in w$.
    To see that harder right-to-left direction is true, 
    suppose $|w| \not\in \mathcal{D}$, then for any $D \in \mathcal{D}$ and $v \in D$, 
    $v \not\sim w$. It is then easy to see that $\chi_D \not\in w$ 
    by discussing which condition for $\sim$ fails. 
    Now since (1) is false, for any $D \in \mathcal{D}$, $\alpha \not\in \Sigma(D)$
    and thus $\dot\lnot\alpha \in \Sigma(D)$ and is a conjunct of $\chi_D$. 
    Thus $\varphi \to \lnot\alpha \in \mathsf{L}_{\mathrm{DTDS}}$. 
    Now we show that (*) $\varphi \to \mathsf{X}(\varphi \lor (\lnot\alpha \land \lnot\beta)$
    is in $\mathsf{L}_{\mathrm{DTDS}}$ as well. 
    Pick any $w \in S^c$ and let $v$ be its only $\to^c$ successor. 
    If $|w| \not\in \mathcal{D}$, then $\varphi \not\in w$ and thus (*) is in $w$. 
    Now consider the $|w| \in \mathcal{D}$ case. Either $|v| \in \mathcal{D}$, 
    in which case $\varphi \in v$, or $|v| \not\in\mathcal{D}$, 
    in which case $\alpha \not\in v$ (since (2) is false) and $\beta \not\in v$ 
    (otherwise $|v|$ would be in $\mathcal{D}$).
    Thus $\varphi \lor (\lnot\alpha \land \lnot\beta) \in v$, 
    and $\mathsf{X}(\varphi \lor (\lnot\alpha \land \lnot\beta)) \in w$. 
    Then (*) is in $w$ again.
    Since (*) is in every MCS, (*) must be a theorem in $\mathsf{L}_{\mathrm{DTDS}}$. 
    Using the extra rule given in Lemma \ref{lem:extra-rule}, 
    $\varphi \to \lnot\mathsf{U}(\alpha, \beta) \in \mathsf{L}_{\mathrm{DTDS}}$. 
    But then, for any $w \in C$, which is in $\mathcal{D}$, 
    $w$ would contain $\lnot\mathsf{U}(\alpha, \beta)$, 
    contradicting that $\mathsf{U}(\alpha, \beta) \in \Sigma(C)$.
\end{proof}

\begin{lemma}\label{lem:truth-lemma}
    For any $\varphi \in \Sigma$ and for any $w \in S^c$, 
    $\mathcal{M}^f, |w| \models \varphi$ iff $\varphi \in w$. 
\end{lemma}
\begin{proof}
    By induction on $\varphi$. The base case and the inductive step
    for $\Box$ and $\mathsf{X}$ are standard as $R^f_\Box$ and $\to^f$
    are the existential lifting of $R^c_\Box$ and $\to^c$ respectively. 

    For $[\mathsf{Agt}]$ and $[i]$, since their corresponding relation 
    in $\mathcal{M}^f$ requires taking a transitive closure, 
    the proof of the inductive steps for them are also similar.
    Take $[i]$ for example. Take any $[i]\varphi \in \Sigma$. 
    Suppose $\mathcal{M}^f, |w| \models [i]\varphi$
    and $w R^c_{[i]} v$. Then $|w| R^f_{[i]} |v|$, 
    and $\mathcal{M}^f, |v| \models \varphi$. 
    By inductive hypothesis, $\varphi \in v$.
    Thus $\varphi$ is in all $R^c_{[i]}$-successors of $w$ 
    and $[i]\varphi$ must be in $w$. 
    Conversely, suppose $\mathcal{M}^f, |w| \not\models [i]\varphi$. 
    Then by semantics and the definition of $R^f_{[i]}$, 
    there is a chain $C_0 = |w| R^e_{[i]} C_1 R^e_{[i]} \dots R^e_{[i]} C_m$ ($m \ge 1$)
    with $\varphi \not\in \Sigma(C_m)$. 
    This means $\dot\lnot\varphi \in \Sigma(C_m)$, and $\dot\lnot[i]\varphi \in \Sigma(C_{m-1})$. 
    Now we show that if $\dot\lnot[i]\varphi \in \Sigma(C_{k})$, then 
    $\dot\lnot[i]\varphi \in \Sigma(C_{k-1})$ as well. 
    Since $C_{k-1} R^e_{[i]} C_k$, there is $u \in C_{k-1}$ and $v \in C_k$ such that 
    $u R^c_{[i]} v$. Since $\dot\lnot[i]\varphi \in \Sigma(C_k)$, 
    $\dot\lnot[i]\varphi \in v$. 
    Then $\dot\lnot[i][i]\varphi \in u$ due to how $R^c_{[i]}$ is defined 
    and that $u, v$ are MCSs. 
    By the $\mathsf{4}$ axiom for $[i]$, $\dot\lnot[i]\varphi \in u$.
    Clearly $\dot\lnot[i]\varphi \in \Sigma$. 
    Thus $\dot\lnot[i]\varphi \in C_{k-1}$. 
    By repeated use of the above conditional, $\dot\lnot[i]\varphi \in \Sigma(C_0)$, 
    and thus is in $w$. Then $[i]\varphi$ must not be in $w$. 

    Finally, to take care of the inductive step of $\mathsf{O}_i$, 
    take any $\mathsf{O}_i\varphi \in \Sigma$. 
    Suppose $\mathcal{M}^f, |w| \models \mathsf{O}_i\varphi$
    and $w R^c_{\mathsf{O}_i} v$. Then $|w| R^f_{\mathsf{O}_i} |v|$ 
    (note that $R^f_{[i]}$ is reflexive), 
    and thus $\mathcal{M}^f, |v| \models \varphi$. 
    By inductive hypothesis, $\varphi \in v$.
    Thus $\varphi$ is in all $R^c_{\mathsf{O}_i}$-successors of $w$ 
    and $\mathsf{O}_i\varphi$ must be in $w$. 
    Conversely, suppose $\mathcal{M}^f, |w| \not\models \mathsf{O}_i\varphi$. 
    Then there is $C, D \in S^f$ such that $|w| R^e_{\mathsf{O}_i} C$,
    $C R^f_{[i]} D$, and $\mathcal{M}^f, D \not\models \varphi$. 
    Recall that the extra closure condition for $\Sigma$ means $[i]\varphi$ and 
    $\dot\lnot[i]\varphi$ are in $\Sigma$ as well.
    Thus by what we have just shown for $[i]$, 
    $\dot\lnot[i]\varphi \in \Sigma(C)$.
    Now there is $w' \in |w|$ and $v \in C$ such that $w' R^c_{\mathsf{O}_i} v$. 
    Since  $\dot\lnot[i]\varphi \in \Sigma(C)$, it is also in $v$, 
    and thus $\dot\lnot\mathsf{O}_i[i]\varphi \in w'$. 
    By axiom (A7), $\dot\lnot\mathsf{O}_i\varphi \in w'$ 
    and thus is in $w$ too. Then $\mathsf{O}_i\varphi \not\in w$.

    The inductive step for $\mathsf{U}$ has been taken care of by Lemma 
    \ref{lem:until-truth-lemma}.
\end{proof}

\subsection{The completeness proof}
We complete the proof of our main theorem, Theorem \ref{thm:completeness}. 
The soundness of is easy to verify. 
Now take any consistent $\varphi$, and let $\Sigma$ be the 
smallest filtration-ready set containing $\varphi$. 
Note that $\Sigma$ is finite, as it can be obtained by first closing $\varphi$
under subformulas, then adding $\mathsf{X}\mathsf{U}(\alpha, \beta)$ for any 
$\mathsf{U}(\alpha, \beta) \in \Sigma$ and $[i]\varphi$ for any $\mathsf{O}_i\varphi \in \Sigma$, 
and finally closing off under $\dot\lnot$. 
Also, let $w$ be a MCS containing $\varphi$. 

Now construct the filtration $\mathcal{M}^f = (S^f, \dots)$ 
according to Definition \ref{def:filtration}.
By Lemma \ref{lem:truth-lemma}, $\mathcal{M}^f, |w| \models \varphi$.
By Lemma \ref{lem:filtration-right-structure}, $\mathcal{M}^f$ is a DTDS Kripke premodel.
To apply Lemma \ref{lem:selective-unravel}, we still need to make sure that 
    \begin{itemize}
        \item for any $\mathsf{X}\varphi \in \Sigma$,
            $\mathcal{M} \models \mathsf{X}\varphi \leftrightarrow \lnot\mathsf{X}\lnot\varphi$, and
        \item for any $\mathsf{U}(\alpha, \beta) \in \Sigma$, 
            $\mathcal{M} \models \mathsf{U}(\alpha, \beta) \leftrightarrow 
                (\alpha \lor (\beta \land \mathsf{X}\mathsf{U}(\alpha, \beta)))$.
    \end{itemize}
For the first point, take any $\mathsf{X}\varphi \in \Sigma$ and any $C \in S^f$. 
Since $\to^f$ is serial, $\mathcal{M}^f, C \models \mathsf{X}\varphi \to \lnot\mathsf{X}\lnot\varphi$. 
Now suppose $\mathcal{M}^f, C \models \lnot\mathsf{X}\lnot\varphi$. 
This means there is a $D$ such that $C \to^f D$ and $\mathcal{M}^f, D \models \varphi$. 
By the definition of $\to^f$ and Lemma \ref{lem:truth-lemma}, 
there is $u \in C$ such that $\lnot\mathsf{X}\lnot\varphi \in u$. 
By ($\mathsf{X}$Func), $\mathsf{X}\varphi \in u$, and since this formula is in $\Sigma$, 
Lemma \ref{lem:truth-lemma} applies, and $\mathcal{M}^f, C \models \mathsf{X}\varphi$.
The second point is easier as if $\mathsf{U}(\alpha, \beta) \in \Sigma$, 
then $\alpha$, $\beta$, and $\mathsf{X}\mathsf{U}(\alpha, \beta)$ are all in $\Sigma$, 
and thus $\mathsf{U}(\alpha, \beta) \leftrightarrow 
(\alpha \lor (\beta \land \mathsf{X}\mathsf{U}(\alpha, \beta)))$, which is an axiom,
is a Boolean combination of formulas in $\Sigma$. 
Clearly Lemma \ref{lem:truth-lemma} can be extended to Boolean combinations of formulas in $\Sigma$.

So we can apply Lemma \ref{lem:selective-unravel} and 
obtain a super-additive DTDS interpreted system $\mathcal{I} = (H, \dots)$
and surjective $\pi: (H \times \mathbb{N}) \to S^f$ such that 
$(h, t)$ and $\pi((h, t))$ satisfy the same formulas in $\Sigma$. 
By surjectivity, we have an $(h, t)$ such that $\pi((h, t)) = |w|$, 
and thus $\mathcal{I}, (h, t) \models \varphi$. 
Finally, we invoke Lemma \ref{lem:convert-to-additive} and \ref{lem:frame-to-model}
to obtain a DTDS interpreted system $\mathcal{I}' = (H', \dots)$ 
and a surjective map $\pi'$ from $H'$ to $H$ such that 
$\mathcal{I}', (h', t')$ and $\mathcal{I}, \pi'((h', t'))$ satisfies the same
formulas. Then, take any $(h', t') \in \pi'^{-1}((h, t))$, 
and we have that $\mathcal{I}', (h', t') \models \varphi$.
In sum, there is a DTDS interpreted system in which $\varphi$ is true.

\subsection{Decidability and complexity}
Given our filtration method, the decidability of $\mathsf{L}_\mathrm{DTDS}$ naturally follows. For any $\varphi \in \mathcal{L}_\mathrm{DTDS}$, let $\Sigma$ be the smallest filtration-ready set of formulas containing $\varphi$. It is not hard to see that $|\Sigma|$ is in $O(|\varphi|)$ where $|\cdot|$ takes length or size. Then, since to identify two worlds in the filtration defined in Definition \ref{def:filtration}, we not only looks the subset of formulas in $\Sigma$ they satisfy, but also the subsets of formulas in $\Sigma$ their $\Box$-neighbours satisfy, the filtrated model may contain at most $O(2^{2^{|\Sigma|}})$ states. Thus, if $\varphi$ is satisfiable (consistent), then it is satisfied on a DTDS Kripke premodel of size $O(2^{2^{|\varphi|}})$ satisfying the two extra properties in Lemma \ref{lem:selective-unravel}. Conversely, due to our work in Section \ref{sec:general-semantics}, if $\varphi$ is satisfiable on a DTDS Kripke premodel satisfying the two extra properties, then it is satisfiable under the intended semantics based on interpreted systems. Thus, to decide if $\varphi$ is satisfiable, we only need to enumerate all possible DTDS Kripke premodels of size at most $O(2^{2^{|\varphi|}})$, check if it satisfies the two extra properties (which obviously does not increase order of space required for computation), and if so, check if it also satisfies $\varphi$. Thus, the satisfiability problem for $\mathsf{L}_\mathrm{DDTS}$ is in 2-EXPSPACE. 

It must be admitted that this is a very rough upper bound; the main merit of this proof is that it does not need automata theory. However, we do believe that proofs in \cite{boudouConcurrentGameStructures2018,halpernComplexityReasoningKnowledge1988} can be adapted to obtain a 2-EXPTIME completeness result.

\section{Conclusion}\label{sec:conclusion}
We have shown how to axiomatize the discrete-time temporal deontic STIT logic based on interpreted systems. Now we mention a few directions that may be fruitful for future investigation. 

First, this logic is different from the logic determined by full discrete-time branching time structures, and an axiomatization of that logic will necessarily involve techniques in \cite{reynoldsAxiomatizationFullComputation2001}. While this is a technically worthwhile project, it is also important to ask conceptually which logic is more suitable for deontic reasoning, or whether they are equally appropriate. 

Second, we believe there should be a way to systematically and automatically axiomatize  logics based on augmenting interpreted systems with synchronous modalities. While it is known that sometimes we may enter the realm of non-axiomatizability \cite{halpern2004complete}, an analogue of Sahlqvist completeness theorem should be within reach. 

Finally, we have not provided any concrete story of how certain actions are allowed while  others are forbidden, nor is there a concrete story of how the norms persist or fail to persist. In other words, $\mathsf{L}_\mathrm{DTDS}$ does not include any reduction of the deontics to other features of the models, or any non-trivial logical principles governing the interaction between the deontic and the temporal. In both \cite{murakami2004utilitarian} and \cite{van2019neutral} the deontic requirements are generated by a utilitarian model while \cite{van2019neutral} also considers temporal operators, and \cite{broersenWhatFailToday2008} uses product frames to generate a logic with non-trivial principles of obligation propagation in time, though only in a single agent setting. In multi-agent and temporal settings, we believe that we must have concrete tokens symbolizing rights that can be transferred between agents or even groups of agents and are kept by agents by default. Petri nets may be useful for such modeling \cite{silenoPetriNetBasedNotation2018}.

\bibliography{sn-bibliography}% common bib file

\end{document}